\documentclass[12pt]{amsart}
\usepackage{latexsym}
\usepackage{amssymb}
\usepackage{amsmath}
\newtheorem{Thm}[subsection]{Theorem}
\newtheorem{Lemma}[subsection]{Lemma}
\newtheorem{Rem}[subsection]{Remark}
\newtheorem{Def}[subsection]{Definition}
\newtheorem{Cor}[subsection]{Corollary}
\numberwithin{equation}{section}% equation numbers within sections
\newcommand{\ds}{\displaystyle}

\newcommand{\ben}{\begin{enumerate}}
\newcommand{\een}{\end{enumerate}}

\newcommand{\bec}{\begin{center}}
\newcommand{\eec}{\end{center}}

\newcommand{\beq}{\begin{equation}}
\newcommand{\eeq}{\end{equation}}

\newcommand{\bdm}{\begin{displaymath}}
\newcommand{\edm}{\end{displaymath}}

\newcommand{\R}{\mathbb{R}}

\newcommand{\N}{\mathbb{N}}

\newenvironment{proofof}[1]{{\sc Proof of #1}}{\quad\lower0.05cm\hbox{$\square$}\medskip}

\newcommand{\bgp}{\bigskip}

\title[Schr\"{o}dinger's Equation and the Balayage]
{Positive Solutions to Schr\"{o}dinger's Equation and the Exponential Integrability of the Balayage}

\author{Michael Frazier}
\address{Mathematics Department, University of Tennessee, Knoxville,
Tennessee 37922} \email{frazier@math.utk.edu}

\author{Igor E. Verbitsky}
\address{Department of Mathematics, University of Missouri, Columbia, Missouri 65211}
\email{verbitskyi@missouri.edu}

\thanks{The second author is supported in part
by NSF grant DMS-1161622.}



\begin{document}

\begin{abstract}
Let $\Omega \subset \R^n$, for $n \geq 2$, be a bounded $C^2$ domain. Let $q \in L^1_{loc} (\Omega)$ with $q \geq 0$.   We give necessary
conditions and matching sufficient conditions, which differ only in the
constants involved, for the existence 
 of very weak solutions to
the boundary value problem $(-\triangle  -q) u =0,  \, \,  u\ge 0 \, \,   \text{on} \,  \, \Omega, \,  u=1  \, \text{on} \, \, \partial \Omega$, and  the related nonlinear problem with quadratic growth in the gradient, $-\triangle u = |\nabla u|^2 + q \,   \text{on} \,  \Omega, \,  u=0 \, \,  \text{on} \, \,  \partial \Omega$.  
We also obtain precise pointwise estimates of solutions up to the boundary. 

A crucial role is played by 
a new ``boundary condition'' on $q$ which is expressed in terms of the exponential integrability on $\partial \Omega$ of the balayage of the measure 
$\delta  q  \,  dx$, 
where $\delta (x) = \text{dist} (x, \partial \Omega)$. This condition is sharp, and 
appears in such a context for the first time. It holds, for example, if $\delta q \, dx$ is a Carleson measure in $\Omega$, or if its balayage  is in $BMO(\partial \Omega)$, with sufficiently small norm. This solves an open problem posed in the literature. 

\end{abstract}

\maketitle \vfill

\eject

\tableofcontents

\section{Introduction}

Let $n \geq 2$ and let $\Omega$ be a bounded $C^2$ domain in $\R^n$.  Let $q$ be a non-negative, locally integrable function on $\Omega$. 
 Our main results give conditions for the existence of positive solutions of the following two problems fundamental to the mathematical theory of the Schr\"{o}dinger operator
$-\triangle  - q$ (see, e.g., \cite{CZ} for $q$ in Kato's class):  

\begin{equation}\label{u0equation} 
\left\{ \begin{aligned}
-\triangle u & = q u +1, \, \,  & u \ge 0 \quad &\mbox{in} \, \,  \Omega, \\
u & = 0 \, \, &\mbox{on} \, \,  \partial \Omega,
\end{aligned}
\right.  
\end{equation} 

\begin{equation}\label{dirichlet} 
\left\{ \begin{aligned}
-\triangle u & = q u, \, \, & u \ge 0 \quad  &\mbox{in} \, \,  \Omega, \\
u & = 1 \, \, &\mbox{on} \, \,  \partial \Omega.
\end{aligned}
\right.  
\end{equation} 

For \eqref{dirichlet}, we also obtain results for $\Omega = \R^n$, $n \ge 3$.   

Our results solve an open problem on the existence of solutions to \eqref{dirichlet}, as well as the corresponding nonlinear problem \eqref{nonlineareqn-1}  with 
quadratic growth in the gradient  discussed below, which was posed in 1999 in \cite{HMV}. 

Equations \eqref{u0equation} and \eqref{dirichlet} have formal solutions as follows.  Let $G(x,y)$ be the Green's function on $\Omega$ associated with the Laplacian $-\triangle$.  Let $G$ denote the corresponding Green's potential operator:
\begin{equation}
  Gf (x) = \int_{\Omega} G(x,y) f(y) \, dy,   \,\,\, x \in \Omega. \label{defngreenpot}
\end{equation}
Let $G_1 =G$ and define $G_j$ inductively for $j \geq 2$ by
\begin{equation}
 G_j (x,y) = \int_{\Omega} G_{j-1} (x,z) G(z,y) \, q(z) \, dz. \label{defGj}
\end{equation}

We define the minimal Green's function associated with the Schr\"{o}dinger operator $-\triangle - q$ to be   
\begin{equation}
 \mathcal{G} (x,y) = \sum_{j=1}^{\infty} G_j (x,y).  \label{defGreenfcnSchr}
 \end{equation}  
The corresponding Green's operator is 
\[  \mathcal{G}f(x)  = \int_{\Omega} \mathcal{G}(x,y) f(y) \, dy .\] 
We let 
\[ u_0 (x) =  \mathcal{G}1(x) = \int_{\Omega} \mathcal{G}(x,y)  \, dy  \] %\label{defu0}
and 
\begin{equation}\label{u1}
u_1 (x) = 1 + \mathcal{G} q (x) = 1 + \int_{\Omega} \mathcal{G} (x,y) \, q(y) \, dy .  
\end{equation}
Then $u_0$ is a formal solution of (\ref{u0equation}) and $u_1$, called the Feynman-Kac gauge in \cite{CZ}, is a formal solution of (\ref{dirichlet}).  The main issue is whether these formal solutions are finite a.e., and consequently solve the corresponding boundary value problems in a certain 
generalized sense.   Problem (\ref{dirichlet}) is more delicate than (\ref{u0equation}) because we must estimate $\mathcal{G} q$ for (\ref{dirichlet}) instead of $\mathcal{G}1$ for (\ref{u0equation}).   

We emphasize that our only {\em a priori} assumptions on the potential $q$ are that $q \in L^1_{loc} (\Omega)$ and $q \geq 0$.  For potentials in Kato's class, both $u_0$ and $u_1$ are finite a.e. if and only if the spectrum of the Schr\"odinger operator is positive on $L^2(\Omega)$, or equivalently (\ref{impliesnormcond2}) below holds for some $\beta \in (0,1)$. In that case, $u_0$ and $u_1$ are uniformly bounded by positive constants both from above and below. This is a consequence of the so-called Gauge Theorem (see, e.g.,  \cite{CZ}), which is no longer true for the general classes of potentials considered in this paper. 

Let $\delta (x) = $ dist $(x, \partial \Omega)$, for $x \in \Omega$.  Let $C^{\infty}_0 (\Omega)$ be the class of $C^{\infty}$ functions with compact support in $\Omega$, and let $L^{1,2}_0 (\Omega)$ be the closure of $C^{\infty}_0 (\Omega)$ with respect to the Dirichlet norm
\[  \Vert f \Vert_{L^{1,2}_0 (\Omega)} = \Vert \nabla f \Vert_{L^2 (\Omega, dx)} .  \]

\begin{Thm}\label{u0thm}
Suppose $\Omega$ is a bounded $C^2$ domain in $\R^n$, for $n \geq 2$, 
and $q \in L^1_{loc}(\Omega)$, $q\ge 0$.  

(i) Suppose there exists $\beta \in (0,1)$ such that
\begin{equation}\label{impliesnormcond2}
\int_{\Omega} h^2 q \, d x \leq \beta^2 \int_{\Omega} \left|\nabla h \right|^2 \, dx  \,\,\, \mbox{for all} \,\,\, h \in C^{\infty}_0 (\Omega).
\end{equation}
Then $u_0 = \mathcal{G} 1 \in L^1 (\Omega, \, dx) \cap L^1 (\Omega, \delta q \, dx)\cap L^{1,2}_0 (\Omega)$, $u_0$ is a positive weak solution of (\ref{u0equation}), and there exist constants $C>0$ depending only on $\Omega$ and $\beta$, and $C_1>0$ depending only on $\Omega$, such that 
\begin{equation}\label{u0upperestimate}
u_0 (x) \leq C_1 \delta (x) e^{C \frac{ G(\delta q)(x)}{\delta (x)}}, \,\,\, \mbox{for all} \,\,\, x \in \Omega.
\end{equation}   

(ii) Conversely, if (\ref{u0equation}) has 
a positive very weak solution $u$, then (\ref{impliesnormcond2})
holds with $\beta =1$ and there exist positive constants $c>0$ and $c_1 >0$ depending only on $\Omega$ such that 
\begin{equation}
 u (x) \geq  c_1 \delta (x) e^{c \frac{ G(\delta q)(x)}{\delta (x)}}  \,\,\, \mbox{for a.e.} \,\,\, x \in \Omega.  \label{u0lowerestimate}
\end{equation} 
 
\end{Thm}

It is easy to see, as a consequence of Theorem \ref{u0thm}, that an analogue of part
(i) of Theorem~\ref{u0thm} with $u_0 = \mathcal{G} f$ 
holds for any bounded measurable function $f$ in place of the
function $1$ on the right hand side of (\ref{u0equation}), and part (ii) for any measurable function $f$ 
bounded below by a positive constant.

Condition (\ref{impliesnormcond2}) was studied originally  by V. G. Maz'ya  for general open sets $\Omega \subset \R^n$ and Borel measures $d \omega$ in 
place of $q dx$, and characterized in terms of capacities associated with 
$L^{1,2}_0 (\Omega)$ (see \cite{M}, Sec. 2.5.2). 

For equation (\ref{u0equation}), the existence of a weak solution under assumption (\ref{impliesnormcond2})  follows by well-known techniques (see, e.g., \cite{DD}).  Also the lower estimate (\ref{u0lowerestimate}) in Theorem \ref{u0thm} is known (see \cite{GrV} and the literature cited there).  What is new here is the upper estimate (\ref{u0upperestimate}), whose proof relies on results in \cite{FNV}.  This upper estimate is, in turn, critical for our results regarding (\ref{dirichlet}).  The more difficult nature of (\ref{dirichlet}) compared to (\ref{u0equation}) is exhibited in our results in two ways: we must consider solutions of (\ref{dirichlet}) in the ``very weak" sense (see Definitions \ref{defveryweak} and \ref{defveryweak2}), and, most importantly, a new condition (\ref{balayagecond}) that controls the behavior of $q$ near $\partial \Omega$ is needed for (\ref{dirichlet}) but not for (\ref{u0equation}).

Let $P(x,y)$ be the Poisson kernel for $\Omega$, and let $P^*$ denote the balayage operator 
(formally adjoint to the Poisson integral) defined by
\[ P^* f (y) = \int_{\Omega} P(x,y) f(x) \, dx , \quad y \in \partial \Omega .\]  
%\label{defbalyage}
Let $d\sigma$ be surface measure on $\partial \Omega$.

\begin{Thm}\label{nonlinthm}
Suppose $\Omega$ is a bounded $C^2$ domain in $\R^n$, for $n \geq 2$, 
and $q \in L^1_{loc}(\Omega)$, $q\ge 0$.   

(i) Suppose there exists $\beta \in (0,1)$ such that
(\ref{impliesnormcond2}) holds and
\begin{equation}\label{balayagecond}
\int_{\partial \Omega} e^{C P^* (\delta q)} \, d \sigma < +\infty
\end{equation}
where $C$ is the constant in (\ref{u0upperestimate}).  Then $u_1 = 1 + \mathcal{G} q$ is a positive very weak solution of (\ref{dirichlet}) with
\begin{equation} \label{u1L1ineq}
 \Vert u_1 \Vert_{L^1 (\Omega, \, dx)} \leq |\Omega| + C_1  \int_{\Omega} e^{C P^* (\delta q)} \, d \sigma ,    
\end{equation}
for some constant $C_1 >0$ depending only on $\Omega$.  Also, there exist positive constants $C_2, C_3$ depending only on $\Omega$ and $\beta$ such that 
\begin{equation} \label{pointwiseestu1}
u_1 (x) \leq C_2 \int_{\partial \Omega} e^{C_3 \int_{\Omega} G(x,y) \frac{P(y,z)}{P(x,z)} q(y) \, dy    } P(x,z) \, d \sigma (z), \,\,\, \mbox{for all} \,\,\,  x \in \Omega.
\end{equation}

(ii)  Conversely, if (\ref{dirichlet}) has a positive very weak solution $u$, then (\ref{impliesnormcond2})
holds with $\beta =1$, (\ref{balayagecond}) holds with the same constant $c$ as in (\ref{u0lowerestimate}), and 
\begin{equation} \label{u1L1ineqconv}
  \int_{\Omega} e^{c P^* (\delta q)} \, d \sigma \leq C_4 \left( \Vert u \Vert_{L^1 (\Omega, \, dx)} + |\partial \Omega| \right) , 
\end{equation}
for some constant $C_4>0$ depending only on $\Omega$.  Moreover, there exist positive constants $c_1, c_2 $ depending only on $\Omega$ such that
\begin{equation} \label{pointwiseestu}
u(x) \geq c_1 \int_{\partial \Omega} e^{c_2 \int_{\Omega} G(x,y) \frac{P(y,z)}{P(x,z)} q(y) \, dy    } P(x,z) \, d \sigma (z) , \,\,\, \mbox{for all} \,\,\,  x \in \Omega.
\end{equation}
  
\end{Thm}

We observe that under the assumptions of Theorem~\ref{nonlinthm} (i), it follows that $u_1\in L_{loc}^{1, 2}(\Omega)$ (see \cite{JMV}, Theorem 6.2). 
If we assume $q \in L^1 (\Omega)$, in addition to (\ref{impliesnormcond2})
 with $\beta \in (0, 1)$, then $u_1-1\in L_0^{1, 2}(\Omega)$, and $u_1$ is a weak solution  to (\ref{dirichlet}), instead of just very weak (see, e.g.,  \cite{AB}). 

We can distinguish condition (\ref{impliesnormcond2}) for Theorem \ref{nonlinthm} from (\ref{balayagecond}) in Theorem \ref{u0thm} via the example $q(x) = a \delta(x)^{-2}$, with $a C<1$ where $C$ is the constant in Hardy's inequality 
%\label{hardy-ineq}
\[ \int_{\Omega} \frac{h^2}{\delta(x)^2}  \, d x \leq C \int_{\Omega} \left|\nabla h \right|^2 \, dx  \,\,\, \mbox{for all} \,\,\, h \in C^{\infty}_0 (\Omega). \]
Then (\ref{impliesnormcond2}) holds, and hence the conclusions of Theorem 1.1 (i) follow for equation (\ref{u0equation}).  However, if (\ref{dirichlet}) had a positive very weak solution $u$, then by (\ref{u1L1ineqconv}), $P^* (\delta q)$ would be exponentially integrable on $\Omega$.  By Jensen's inequality, we would then have $\int_\Omega \delta(x) q(x) dx < \infty$, which fails for $q(x) = a \delta(x)^{-2}$.  

We remark that the additional condition $\int_\Omega \delta(x) q(x) dx < \infty$, or equivalently $G q <+\infty$ a.e., combined  with (\ref{impliesnormcond2})  for any $\beta\in (0, 1)$,  
is generally not enough (unless $n=1$) to ensure that $u_1$ 
is a very weak solution to (\ref{dirichlet}). 

Theorem \ref{nonlinthm} leads to conditions for the existence of a very weak solution to (\ref{dirichlet}) in terms of Carleson measures and BMO.  For a measure $\mu$ on $\Omega$, define the Carleson norm of $\mu$ by 
\[ \Vert \mu \Vert_C = \sup_{r>0, x \in \partial \Omega} r^{1-n} \mu (\{y \in \Omega: |y-x| <r  \} ) . \]

For $f \in L^1 (\partial \Omega, \, d \sigma)$, define $U_r(x) = \{y \in \partial \Omega : |y-x|<r \}$ and 
\[ \Vert f \Vert_{BMO(\partial \Omega)} = \sup_{r >0, x \in \partial \Omega} |\sigma(U_r (x))|^{-1} \int_{U_r(x)} |f - f_{U_r(x)}| \, d \sigma,  \]
where $f_{U_r(x)} = |\sigma(U_r (x))|^{-1} \int_{U_r (x)} f\, d \sigma$ is the average of $f$ on $U_r (x)$.

\begin{Cor}\label{carl}
Let $\Omega \subset \R^n$ be a bounded $C^2$ domain, for $n \geq 2$, 
and let $q \in L^1_{loc}(\Omega)$, $q\ge 0$.   Suppose (\ref{impliesnormcond2}) holds for some $\beta \in (0,1)$. Then there exist $\epsilon_1, \epsilon_2 >0$, depending only on $\Omega$ and $\beta$ such that if
  
\vspace{0.1in}

(A) $\Vert P^* (\delta q ) \Vert_{BMO (\partial \Omega)} < \epsilon_1$,

or

(B) $\Vert \delta q \, dx \Vert_C < \epsilon_2$,

\vspace{0.1in} 
\noindent{then} $u_1 \in L^1 (\Omega, dx)$ and $u_1 $ is a positive very weak solution of (\ref{dirichlet}).

\end{Cor}

For the case $\Omega = \R^n$, $n \ge 3$, we denote by $I_2 f = (-\triangle)^{-1} f$ the Newtonian potential of $f$: 
$$
I_2 f(x) = c_{n} \int_{\R^n} \frac{f(y) \, dy}{|x-y|^{n-2}}, 
\quad x \in \R^n, 
$$
where $c_{n}$ is a positive normalization constant.  Let $G(x,y) = c_n |x-y|^{2-n}$ be the kernel of $I_2$.  

\begin{Thm}\label{riesz-est} Let $n \ge 3$. 

(i) Suppose there exists $\beta \in (0,1)$ such that
\begin{equation}\label{trace-1}
\int_{\R^n} h^2 q \, d x \leq \beta^2 \int_{\R^n} \left| \nabla h \right|^2 \, dx \,\,\, \mbox{for all} \,\,\, h \in C^{\infty}_0 (\R^n),
\end{equation}
and
\begin{equation}\label{finite-int}
\int_{\R^n} \frac{q(y) \, dy}{(1 + |y|)^{n-2}} < +\infty.
\end{equation}
Then $u_1= 1 + \mathcal{G} q$ is a positive minimal solution (in the distributional sense) to 
\begin{equation}\label{eqnRn}
\left\{ \begin{aligned}
-\triangle u  = q u \, \,  & \, \, \textrm{on} \, \,  \R^n, \\
\lim \inf_{x \rightarrow \infty} u(x) =1 . & 
\end{aligned}
\right.  
\end{equation} 
Also,
\begin{equation}\label{upper-estRn}
u_1(x) \le \, e^{C \,  I_2 q(x)},  \,\,\, \mbox{for all} \,\,\, x \in \R^n,   
 \end{equation}
where $C$ depends only on $\beta$ and $n$. 
 
(ii) Conversely, if there is a positive (distributional) solution $u$ of (\ref{eqnRn}), then (\ref{trace-1}) holds with $\beta=1$, (\ref{finite-int}) holds,  and 
\begin{equation}\label{lower-estRn}
u(x) \ge \, e^{c \, I_2 q(x)},  \,\,\, \mbox{for all} \,\,\, x \in \R^n,  
 \end{equation}
where $c$ depends only on $ n$.  
 \end{Thm} 
  
Condition (\ref{trace-1}) is the so-called trace inequality which expresses the continuous imbedding of  
$L_0^{1, 2}(\R^n)$ into $L^2(\R^n, q \, dx)$. The class of functions $q$ (or more generally measures $\omega$) 
such that (\ref{trace-1}) holds is well understood, and several characterizations are known (see \cite{AH}, \cite{M}, and the literature cited there).   
  
Theorems \ref{u0equation}, \ref{nonlinthm}, and \ref{riesz-est} are the model  cases of more general results for wider classes of operators, including fractional Laplacians, and domains $\Omega$ (Lipschitz and NTA domains), 
as well as more general right-hand sides and boundary data, that we plan to address in a forthcoming paper.  

The Feynman-Kac gauge $u_1$ is  closely related, via a formal substitution $v = \log
u_1$, to a generalized solution of the nonlinear boundary value problem 
with quadratic growth in the gradient: 
\begin{equation}\label{nonlineareqn-1} 
\left\{
\begin{aligned}
-\triangle v & = |\nabla v|\,^2 + q \, \, & \mbox{in} \, \,  \Omega \\
v & = 0 \quad & \mbox{on} \, \,  \partial \Omega .
\end{aligned}
\right. 
\end{equation} 

However, it is well known that the relation between (\ref{dirichlet}) and (\ref{nonlineareqn-1}) is not as simple as the formal substitution suggests 
(see \cite{FM2}).  Nevertheless, we obtain the following result.

\begin{Thm}\label{riccatithm}  Suppose $\Omega \subset \R^n$ is a bounded $C^2$ domain, where $n \geq 2$, and  $q \in L^1_{loc}(\Omega)$, $q\ge 0$.    

(i)  Suppose there exists $\beta \in (0,1)$ such that (\ref{impliesnormcond2}) holds, and (\ref{balayagecond}) holds.  Then $v= \log u_1 $ is a very weak solution of (\ref{nonlineareqn-1}) with $v \in L^{1,2}_{loc} (\Omega)$.  
 
(ii) Conversely, if (\ref{nonlineareqn-1}) has a very weak solution in $L^{1,2}_{loc} (\Omega)$, then (\ref{impliesnormcond2})
holds with $\beta=1$, and (\ref{balayagecond}) holds with some small constant $c = c (\Omega) >0$.  
\end{Thm}
 
A similar problem for the superquadratic equation
 \[ - \triangle v = |\nabla v|^s  + q,\]
with $s>2$, was solved in \cite{HMV}, where a thorough discussion   of such problems and more details can be found. We remark 
that no additional condition like \eqref{balayagecond} is required for $s>2$.  Theorem \ref{riccatithm} resolves the case $s=2$, which was stated as an open problem in \cite{HMV}.

Regarding solutions to (\ref{nonlineareqn-1}), we refer also to Ferone and Murat \cite{FM1} where the existence of finite energy solutions  $v \in L^{1, 2}_0(\Omega)$ is  proved  for $q \in L^{\frac{n}{2}}(\Omega)$ ($n \ge 3$), with sufficiently small norm; in that case $u_1-1=e^v-1 \in L^{1, 2}_0(\Omega)$. 
In  \cite{FM3}, these results are extended to $q \in L^{\frac{n}{2}, \infty}(\Omega)$.  
(See also \cite{ADP}, \cite{AB} where the existence of such solutions is obtained for $q \in L^1(\Omega)$ satisfying  (\ref{impliesnormcond2})
 with $\beta \in (0, 1)$.) Clearly, for $q \in L^{\frac{n}{2}, \infty}(\Omega)$, the assumptions of 
Corollary~\ref{carl}, and hence Theorem~\ref{riccatithm},  are satisfied; that is,  \eqref{impliesnormcond2} holds, and 
$\delta q$ is a Carleson measure, which yields \eqref{balayagecond}.

In Section \ref{weak}, we discuss very weak solutions for Schr\"{o}dinger equations.  The proofs of Theorems \ref{u0equation}, \ref{nonlinthm}, and \ref{riesz-est} are given in Section \ref{applications}.  In Section \ref{riccati}, we discuss the nonlinear equation \eqref{nonlineareqn-1} and prove Theorem \ref{riccatithm}, using techniques from potential theory.

We would like to thank Fedor Nazarov for valuable conversations related to the content of this paper, which is a continuation and application of \cite{FNV}.

\bgp

\section{Very Weak Solutions}\label{weak}

\bgp 

Let $\Omega \subset \R^n$ be a bounded $C^2$ domain with Green's function $G(x,y)$, where $n \geq 2$.  Recall that $\delta(x) = \text{dist}(x, \partial \Omega)$.  We will use the following well-known estimates repeatedly:
\begin{equation}\label{green-smootha}
G (x,y) \approx \frac{\delta(x) \, \delta(y)}{|x-y|^{n - 2} (|x-y| + \delta (x) + \delta(y))^2 },  \quad n \ge 3, 
\end{equation}
\begin{equation}\label{green-smoothb}
G (x,y) \approx \ln \left (1+ \frac{\delta(x) \, \delta(y)}{|x-y|^{2}} \right),  \quad n =2, 
\end{equation}
for all $x, y \in \Omega$, where ``$\approx$" means that the ratio of the two sides is bounded above and below by positive constants depending only on $\Omega$  
(see \cite{Wid}, \cite{Zh} for $n \ge 3$; \cite{CZ}, Theorem 6.23 for $n=2$). 

Estimates \eqref{green-smootha}, \eqref{green-smoothb} yield   a  
cruder upper estimate
\begin{equation}\label{green-smoothc}
G (x,y) \le C \frac{\delta(x)}{|x-y|^{n - 1}},  \quad n \ge 2 ,   
\end{equation}
for all $x, y \in \Omega$. This is obvious if $n \ge 3$; for $n=2$, notice that 
\[
 \ln \left (1+ \frac{\delta(x) \, \delta(y)}{|x-y|^{2}} \right) \le \frac{\delta(x) \, \delta(y)}{|x-y|^{2}} .
\]
Hence, for  $\delta (y) \le 2 |x-y|$, we have,  
\[
G(x, y) \le C \frac{\delta(x)}{|x-y|}. 
\]
For $\delta (y) > 2 |x-y|$, using 
the inequality $\delta(y)\le |x-y|+ \delta(x)$, we see that $|x-y|< \delta (x)$ and 
$\delta(y) < 2 \delta(x)$. Hence, in this case, 
\[
G(x, y) \le C  \ln \left (1+ \frac{\delta(x) \, \delta(y)}{|x-y|^{2}}\right) \le 
C  \ln \left (1+ \frac{2 \delta(x)^2}{|x-y|^{2}}\right) \le C \frac{\delta(x)}{|x-y|} , 
\]
which verifies \eqref{green-smoothc} for $n = 2$. 

The preceding estimates yield 
\begin{equation}\label{G1est}
G1 (x) = \int_{\Omega} G(x,y) \, dy \approx \delta (x) , \quad n \ge 2,   
\end{equation}
for all $x \in \Omega$. Indeed, the lower bound $G 1(x) \ge c \, \delta(x)$   follows from the well-known estimate 
$G(x, y)\ge c \, \delta(x) \delta(y)$, which is an obvious consequence of  \eqref{green-smootha}, \eqref{green-smoothb}. The upper bound in \eqref{G1est} 
follows by integrating both sides 
of \eqref{green-smoothc} with respect to $dy$ over a ball $B(x, R)$ with  
$R=\text{diam} (\Omega)$ so that $\Omega \subset B(x, R)$:
\[
G1 (x) \le C \, \delta(x) \int_{B(x, R)} \frac{dy}{|x-y|^{n - 1}} = C_1 \,  \delta(x) .
\]

Our first goal is to define a very weak solution for Schr\"{o}dinger equations.  We begin by defining very weak solutions for Poisson's equation with Dirichlet boundary conditions.  We will use the class of test functions
\[  C^2_0 (\overline{\Omega})= \{ h \in C^2 (\overline{\Omega}): h=0 \,\,\mbox{on} \,\, \partial \Omega \}. \]

\begin{Def}\label{defveryweaksoln} Suppose $f \in L^1 (\Omega, \delta dx)$.  A function $u \in L^1(\Omega, dx)$ is a very weak solution of the Dirichlet problem 
\begin{equation}\label{poisson} 
\left\{ \begin{aligned}
-\triangle u & = f\,\,  &\mbox{in} \, \,  \Omega, \\
u & = 0\,\, &\mbox{on} \, \,  \partial \Omega
\end{aligned}
\right.  
\end{equation}
if 
\begin{equation}\label{weak-defn} 
-\int_\Omega u \, \triangle h \,  dx = \int_\Omega h \, f \, dx,
\end{equation}
for all $h \in C^2_0 (\overline{\Omega})$.   
\end{Def}

The following lemma concerning the existence and uniqueness of very weak solutions is well known (see \cite{BCMR}, Lemma 1). 
For convenience we supply a simple proof which shows additionally that the weak solution is given by the Green's potential $G f$, defined by (\ref{defngreenpot}).  

\begin{Lemma}\label{lemma-weak}
(i) Let $f \in L^1 (\Omega, \delta dx)$. Then there exists a unique very weak solution $u \in L^1(\Omega, dx)$ of (\ref{poisson}) given by 
$u=G f$. 

(ii) If $f\ge 0$ a.e. and $Gf(x_0)<+\infty$ for some $x_0 \in \Omega$, then $f \in L^1 (\Omega, \delta dx)$ and $u=Gf\in L^1(\Omega, dx)$ is a very weak solution of  (\ref{poisson}).  
\end{Lemma}
\begin{proof} (i) The proof of uniqueness follows \cite{BCMR}. Suppose both $v$ and $w$ are weak solutions of (\ref{poisson}). Let $\phi \in C^\infty_0(\Omega)$ and let $h = G \phi$.  Then $h \in C^2_0 (\overline{\Omega})$ and $- \triangle h = \phi $ on $\Omega$.  Consequently 
$$
\int_\Omega (v-w) \, \phi \ dx = -\int_\Omega (v-w) \, \triangle h \, dx =0, 
$$
by (\ref{weak-defn}).  Since this equation holds for every $\phi \in C^\infty_0(\Omega)$, we obtain $v=w$. 

Next we prove that if $f \in L^1 (\Omega, \delta dx)$ then $u = G f$ is a weak solution. Without loss 
of generality we may assume that $f\ge 0$. By Fubini's theorem and the symmetry of $G$,  
\[ ||u||_{L^1(\Omega, dx)} =  \int_\Omega \int_\Omega G(x,y) \, f(y) \, dy \, dx \]
\[  = \int_\Omega G1 (y) \, f(y) \, dy \leq C  \int_\Omega \delta (y) \, f(y) \, dy  < +\infty, \]
by (\ref{G1est}).  

Let $f_k \in C^\infty_0(\Omega)$ be a sequence of nonnegative functions such that $||f- f_k||_{L^1(\Omega, \delta dx)} \to 0$ as $k \to +\infty$. 
Denote by $u_k=G f_k$ the solution to (\ref{poisson}) with $f_k$ in place of $f$. By Green's theorem, 
\begin{equation}\label{green-1} 
 -\int_\Omega u_k \, \triangle h \, dx = \int_\Omega h \, f_k \, dx, 
 \end{equation}
 for every $h \in C^2_0 (\overline{\Omega})$. Note that 
 $$
 ||(u-u_k)  \triangle h||_{L^1(\Omega, dx)} \le ||\triangle h||_{L^\infty(\Omega)} 
 ||u-u_k||_{L^1(\Omega, dx)},  
 $$
 where by Fubini's theorem 
 $$
  ||u-u_k||_{L^1(\Omega, dx)} = ||G (f-f_k)||_{L^1(\Omega, dx)} \le  C  ||(f-f_k) \delta ||_{L^1(\Omega, dx)} 
  $$
  $$\le C \, ||f-f_k||_{L^1(\Omega, \delta dx)} \to 0. 
 $$
Note that since $h \in C^2_0 (\overline{\Omega})$, we have $|h(x)| \leq C \delta (x)$.  Hence, passing to the limit as $k \to +\infty$ on both sides of (\ref{green-1})  proves that  $u=G f$ is a very weak solution. This proves statement (i) of Lemma \ref{lemma-weak}. 
 
 To prove statement (ii), assume that  $Gf(x_0)<+\infty$ for some $x_0\in \Omega$, where $f \ge 0$ a.e. Since $u=Gf$ is superharmonic in $\Omega$ (see e.g. Theorem 3.3.1 in [AG]), it follows by the mean value inequality that 
 $$
 \frac{1}{|B(x_0,r)|}\int_{B(x_0, r)} Gf(x) \, dx \le Gf(x_0)<+\infty, 
 $$
 for some ball $B(x_0,r)$ such that $0<r<\frac 1 2 \delta(x)$. 
 By Fubini's theorem, 
 $$
 \int_{B(x_0, r)} Gf(x) \, dx = \int_\Omega G \chi_{B(x_0,r)}(y)  f(y) \, dy. 
 $$
 Since $G \chi_{B(x_0,r)}(y) \ge C \, \delta(y)$ for 
 all $y \in \Omega$, it follows that $f \in L^1(\Omega, \delta dx)$. 
 Thus by statement (i), $u=Gf\in L^1(\Omega, dx)$ is a very weak solution of  (\ref{poisson}).  
 \end{proof} 
 
\begin{Rem}\label{Remarkmeasure} We can extend Definition \ref{defveryweaksoln} and Lemma \ref{lemma-weak} to the case where $f$ is replaced with a signed Radon measure $\omega$ on $\Omega$ such that $\int_{\Omega} \delta \, d \omega < \infty$.  In this case, we say that $u \in L^1 (\Omega, \, dx)$ is a very weak solution of 
\begin{equation}\label{poissonmeasure} 
\left\{ \begin{aligned}
-\triangle u & = \omega \,\,  &\mbox{in} \, \,  \Omega, \\
u & = 0\,\, &\mbox{on} \, \,  \partial \Omega
\end{aligned}
\right.  
\end{equation}
if 
\[ -\int_\Omega u \, \triangle h \,  dx = \int_\Omega h \, d \omega,\]
for all $h \in C^2_0 (\overline{\Omega})$.  Then by Theorem 1.2.2 in \cite{MV}, $u (x) = G \omega  (x) = \int_{\Omega} G(x,y) \, d \omega (y)$ is the unique very weak solution of (\ref{poissonmeasure}).  For future reference in \S \ref{riccati}, we note that the proof of Theorem 1.2.2 in \cite{MV} shows that if $\int_{\Omega} \delta \, d \omega < \infty$, then $G \omega \in W^{1,p}_{loc} (\Omega)$ for $1 \leq p < n/(n-1)$.
\end{Rem}
  
We now use the above definition of very weak solutions of the Poisson equation to define very weak solutions of the Schr\"{o}dinger equation.  For the following definition, and subsequent lemma, we do not require $q \geq 0$.

\begin{Def}\label{defveryweak} Let $q \in L^1_{loc}(\Omega, dx)$  and let 
$f \in L^1(\Omega, \delta(x) dx)$. 
A function $u \in L^1(\Omega, dx)\cap L^1(\Omega,  \delta(x) |q(x)| dx)$ is a very weak solution to the Schr\"{o}dinger equation 
\begin{equation}\label{schro-1} 
\left\{ \begin{aligned}
-\triangle u & = q \, u + f\,\,  &\mbox{in} \, \,  \Omega, \\
u & = 0\,\, &\mbox{on} \, \,  \partial \Omega, 
\end{aligned}
\right.  
\end{equation}
if 
\begin{equation}\label{weak-schro} 
-\int_\Omega u \, \triangle h \,  dx = \int_\Omega h \, u  \, q \, dx +
 \int_\Omega h \, f \, dx ,
\end{equation}
for all $h \in C^2_0 (\overline{\Omega})$. 
\end{Def}

Formally, applying the Green's operator $G$ to both sides of the equation $-\triangle u = qu + f$  yields the integral equation 
\begin{equation}\label{weak-int2}
u(x)  = G(qu+f)(x) = \int_{\Omega} G(x,y) u(y) q(y) \, dy + \int_{\Omega} G(x,y) f(y) \, dy. 
\end{equation}
By a solution of (\ref{weak-int2}) we mean a function $u$ such that $u$ and $G(qu+f)$ are finite and equal a.e.
The relationship between very weak solutions of (\ref{schro-1}) and solutions of (\ref{weak-int2}) is made clear by the following lemma.  

\begin{Lemma}\label{integeqnequiv}
Suppose $q \in L^1_{loc}(\Omega, dx), f \in L^1(\Omega, \delta(x) dx)$, and $u \in L^1(\Omega, dx)\cap L^1(\Omega,  \delta(x) |q(x)| dx)$.  Then $u$ is a very weak solution to the Schr\"{o}dinger equation (\ref{schro-1}) if and only if $u$ 
is a solution to the integral equation  $u = G(qu+f)$.
\end{Lemma}

\begin{proof}
Suppose $h \in C^2_0 (\overline{\Omega})$.  By our assumptions, $qu \in L^1 (\Omega, \delta \, dx)$.  Hence by Lemma (\ref{lemma-weak}), 
\[ 
- \int_{\Omega} G(qu) \triangle h \, dx = \int_{\Omega}  qu h\, dx,  \qquad -\int_{\Omega} G(f) \triangle h  \, dx =  \int_{\Omega}  f h\, dx . 
\]

If we assume $u= G(qu+f) $ a.e., then 
\[  - \int_{\Omega} u \triangle h \, dx = - \int_{\Omega} G(qu + f) \triangle h \, dx = \int_{\Omega} (qu+f) h \, dx, \]
for all $h \in C^2_0 (\overline{\Omega})$, so $u$ is a very weak solution of (\ref{schro-1}).  Conversely, suppose $u$ is a very weak solution of (\ref{schro-1}).
For any $\phi \in C^{\infty}_0 (\Omega)$, then $h = G\phi$ satisfies $h \in C^2_0 (\overline{\Omega})$ and $- \triangle h = \phi$.  Hence 
\[  \int_{\Omega} u \phi \, dx = - \int_{\Omega} u \triangle h \, dx =  \int_{\Omega} h (qu+f) \, dx \]
\[ = - \int_{\Omega} G(qu+f) \triangle h \, dx =  \int_{\Omega} G(qu+f) \phi \, dx . \] 
Since $\phi \in C^{\infty}_0 (\Omega)$ is arbitrary, $u = G(qu+f) $ a.e.
\end{proof}

We now return to our standing assumption that $q \geq 0$.  The following Corollary will be useful.

\begin{Cor}\label{ident}
Suppose $f \in L^1(\Omega, \delta dx)$ and $f \geq 0$.  Suppose $u \geq 0$ satisfies $u = G(qu +f)$.  
Then $u \in L^1(\Omega, dx) \cap  L^1(\Omega, \delta q dx)$.

\end{Cor}

\begin{proof}
By assumption, $u< \infty$ a.e.  In particular, $u(x_0) < \infty$ for some $x_0 \in \Omega$.  Since $Gf\geq 0$, we have that $G(qu) (x_0) < \infty$. By Lemma \ref{lemma-weak} (ii), we have $qu \in L^1 (\Omega, \delta \, dx)$, or $u \in L^1 (\Omega, \delta q\, dx)$. 
 
We integrate the equation $u = G(qu+f)$ over $\Omega$.  Since all terms are nonnegative, Fubini's theorem gives 
\[ \int_\Omega u(x) dx = \int_\Omega u(x) \, G1(x)  \, q(x) \, d x +  \int_\Omega G1(x) \,  f(x) dx \]
\[ \approx \int_\Omega u(x) \, \delta(x)  \, q(x) \, d x +  \int_\Omega \delta(x) \,  f(x) dx < \infty   \]
since $u \in L^1 (\Omega, \delta q \, dx)$.  Hence $u \in L^1 (\Omega, dx)$.
\end{proof}

\begin{Lemma}\label{finiteimpliessoln}
Suppose $f \geq 0$ and $\mathcal{G}(f) < \infty$ a.e.  Then $u = \mathcal{G} f $ satisfies $u \in L^1 (\Omega, dx) \cap L^1 (\Omega, \delta q \, dx), u (x)= G(qu+f)(x)$ for all $x \in \Omega$, and $u$ is a very weak solution of (\ref{schro-1}).
\end{Lemma}

\begin{proof}
We first observe that $G f \leq \mathcal{G}f < +\infty$ a.e., and hence $f \in L^1(\Omega, \delta dx)$ and $Gf \in L^1 (\Omega, \, dx)$, by Lemma \ref{lemma-weak} (ii).  Hence  $Gf$ is finite a.e.  Note that for $j \geq 2$,
\[ \int_{\Omega} G_j (x,y) f(y) \, dy = \int_{\Omega} \int_{\Omega} G(x,z) G_{j-1} (z,y) q(z) \, dz \, f(y) \, dy  \]
\[ = \int_{\Omega} G(x,z) q(z) \int_{\Omega} G_{j-1} (z,y) f(y) \, dy \, dz\]
\[ = G \left( q \int_{\Omega} G_{j-1} (\cdot,y) f(y) \, dy \right) (x)   \] 
by Fubini's theorem. Hence 
\[  \mathcal{G} f(x) = Gf(x) + \sum_{j=2}^{\infty} \int_{\Omega} G_j (x,y) f(y) \, dy \]
\[ =  Gf(x) +  G \left( q \int_{\Omega} \sum_{j=2}^{\infty} G_{j-1} (\cdot,y) f(y) \, dy \right) (x)  = G (f + q \, \mathcal{G}f) (x), \]
or $u (x)= G(qu + f)(x) $, for all $x \in \Omega$.  Since $u$ and $Gf$ are finite a.e., so is $G(qu)$. Hence by Corollary \ref{ident}, $u \in L^1 (\Omega, dx) \cap L^1 (\Omega, \delta q \, dx)$.  By Lemma \ref{integeqnequiv}, $u$ is a very weak solution of (\ref{schro-1}).
\end{proof}

Positive very weak solutions of the Schr\"{o}dinger equation are in general not unique (see \cite{Mur}).  However, if $f \geq 0$ and (\ref{weak-int2}) has a nonnegative solution, then $\mathcal{G}$ is the minimal solution, in the sense that if $u\geq 0$ satisfies (\ref{weak-int2}) then $\mathcal{G} f(x) \leq u(x)$ for a.e. $x$.  To see this fact, define $G_j f(x) = \int_{\Omega} G_j (x,y) f(y) \, dy$ for $G_j(x,y)$ defined by (\ref{defngreenpot}) and (\ref{defGj}), and define $Tg = G(g q)$.  In the proof of the previous lemma, we showed that $G_j f = T(G_{j-1} f)$. Hence, substituting $G(uq+f) = Tu + Gf$ for $u$ repeatedly,
\[  u =  Tu + Gf = T( Tu + Gf) + Gf \]
\[ = T^2 u + T(Gf) + Gf = T^2 u + G_2 f + Gf. \]
Iterating, we obtain $u = T^k u + \sum_{j=1}^k G_j f$, and letting $k \rightarrow \infty$ shows that $u \geq \mathcal{G} (f)$.
Hence $\mathcal{G} (f)$ is called the \textit{minimal very weak solution} of (\ref{schro-1}).  Thus, the only issue regarding the existence of a very weak solution of (\ref{schro-1}) is whether $\mathcal{G} (f) < \infty $ a.e. 

We adapt Definition \ref{defveryweak} of a very weak solution to the case of non-zero boundary conditions.   If $g \in L^1 (\partial \Omega, \, d \sigma)$, then $P(g)$, the Poisson integral of $g$, defined $P(g) (x) = \int_{\partial \Omega} P(x,y) g(y) \, d \sigma (y)$, is harmonic on $\Omega$ and has boundary values $g(y)$ $\sigma$-a.e. The following definition does not require $q \geq 0$.

\begin{Def}\label{defveryweak2} Let $q \in L^1_{loc}(\Omega, dx), f \in L^1(\Omega, \delta(x) dx), g \in L^1 (\partial \Omega, \, d \sigma)$, and $q P(g) \in L^1 (\delta \, dx)$.   
A function $u \in L^1(\Omega, dx)\cap L^1(\Omega,  \delta(x) |q(x)| dx)$ is a very weak solution of the Schr\"{o}dinger equation 
\begin{equation} \label{defnveryweakinhom}  \left\{ \begin{aligned}
-\triangle u & = q \, u + f\,\,  &\mbox{in} \, \,  \Omega, \\
u & = g \,\, &\mbox{on} \, \,  \partial \Omega, 
\end{aligned}
\right.  
\end{equation}
if $u = v + P(g)$, where $v$ is a very weak solution of 
\[ \left\{ \begin{aligned}
-\triangle v & = q \, v + f	 + q P(g)\,\,  &\mbox{in} \, \,  \Omega, \\
v & = 0 \,\, &\mbox{on} \, \,  \partial \Omega . 
\end{aligned}
\right.  \]
\end{Def}

Our definition is not entirely standard, but it is  equivalent to the standard definition (see e.g., Definition 1.1.2 in \cite{MV}) since both of them result in  the integral representation $u=G(qu) +G(f) +P(g)$.

In the case of (\ref{dirichlet}), we have $f=0$ and $g=1$ in (\ref{defnveryweakinhom}).  Then $P(1)=1$, so any very weak solution of  (\ref{dirichlet}) has the form $u= v+1$, where $v$ is a very weak solution of $-\triangle v = qv + q$ on $\Omega$, $v=0$ on $\partial \Omega$.
If we assume $q \in L^1 (\Omega, \delta \, dx)$ then Lemma \ref{integeqnequiv} gives that $v \in L^1 (\Omega, dx) \cap L^1 (\Omega, \delta |q| \, dx)$ is a very weak solution of $-\triangle v = qv + q$ on $\Omega$, $v=0$ on $\partial \Omega$ if and only if $v$ is a solution of the integral equation $v = G(qv + q)$.  Hence $u \in  L^1 (\Omega, dx) \cap L^1 (\Omega, \delta |q| \, dx)$ is a very weak solution of (\ref{dirichlet}) if and only $u$ is a solution of the integral equation $u= 1 + G (qu)$.  

We now return to the assumption that $q \geq0$.  As for Lemma \ref{finiteimpliessoln}, the only issue regarding whether the formal solution $u_1$ in (\ref{u1}) yields a very weak solution to (\ref{dirichlet}) is whether the expression in (\ref{u1}) is finite a.e.  

\begin{Lemma}\label{finiteimpliessoln2}
Suppose  $u_1 = 1 + \mathcal{G}(q) < \infty$ a.e.  Then $u_1$ is a very weak solution of (\ref{dirichlet}).
\end{Lemma}

\begin{proof} 
By Lemma \ref{finiteimpliessoln}, $v = \mathcal{G} (q) $ is a very weak solution of $-\triangle v = qv + q$ on $\Omega$, $v=0$ on $\partial \Omega$.  
\end{proof} 

In fact, $u_1$ is the minimal positive weak solution of (\ref{dirichlet}).  To see this fact, suppose $u \geq 0$ is a positive weak solution of (\ref{dirichlet}).  The equation $u = 1 + G(qu)$ shows that $u \geq 1$.  Hence $v = u-1$ is a positive solution to the integral equation $v = G(qv + q)$, hence a positive very weak solution to (\ref{schro-1}) with $f=q$.  By the minimality of the positive solution $ \mathcal{G} (q)$ of (\ref{schro-1}) with $f=q$, we have $\mathcal{G} (q)  \leq v$, and hence $u_1 = 1 + \mathcal{G} (q) \leq 1+ v = u$.

\bgp

\section{Positive Solutions to Schr\"{o}dinger Equations}\label{applications}

\bgp
   
In Lemma \ref{integeqnequiv}, we reduced the solution (in the very weak sense) of (\ref{schro-1}) to the associated integral equation $u = G(qu) + Gf $.  We define the integral operator $T$ by  
\begin{equation}\label{operator-tf}
T f (x) =G(fq)(x) = \int_\Omega G (x,y) \, f(y) q(y) \, dy, \quad x \in \Omega.   
\end{equation}  
Our first step in proving Theorem \ref{u0thm} is to relate condition (\ref{impliesnormcond2}) to the norm of $T$ on $L^2 (\Omega, q \, dx)$.

\begin{Lemma}\label{leastconst} Suppose $\Omega \subset \R^n$, for $n \ge 1$,  
is a bounded $C^2$ domain. Then $T$ maps $L^2 (q \, dx)$ to itself boundedly if and only if (\ref{impliesnormcond2}) holds for some $\beta$, and $||T||_{L^2(\Omega, q \, dx) \to L^2(\Omega, q \, dx)}=\beta^2$, where $\beta$ is the least constant in (\ref{impliesnormcond2}). 
\end{Lemma} 

\begin{proof} Recall that $L^{1,2}_0 (\Omega)$ is the homogeneous Sobolev space of order 1, that is, the closure of $C^{\infty}_0 (\Omega)$ with respect to the Dirichlet norm 
$\Vert \nabla f \Vert_{L^2 (\Omega, dx)}$.  The dual of $L^{1,2}_0$ is isometrically isomorphic to the space $L^{-1, 2} (\Omega)$ (and vice versa).  For $f \in L^1 (\Omega, dx)$ (or more generally a finite signed measure), we have 
\begin{equation}
\Vert f \Vert_{L^{-1, 2} (\Omega)}^2 = \int_{\Omega} |\nabla G (f)|^2 \, dx = \int_{\Omega} f Gf \, dx \label{dual-Sob11}
\end{equation}
(see \cite{L}, Sec. I.4, Theorem 1.20).  Also note that by duality, the inequality 
\begin{equation}\label{firstSobineq}
\Vert f q \Vert_{L^{-1,2} (\Omega)} \leq \alpha \Vert f \Vert_{L^2 (\Omega, q\, dx)},  \,\,\, \mbox{for all} \,\,\, f\in L^2 (\Omega, q \, dx)
\end{equation}
is equivalent to the inequality
\begin{equation} \label{secSobineq} 
\Vert h\Vert_{L^{2} (\Omega, q\, dx)} \leq \alpha \Vert h \Vert_{L^{1,2}_0 (\Omega)}  \,\,\, \mbox{for all} \,\,\, h \in C^{\infty}_0 (\Omega) .
\end{equation}
For example, if (\ref{firstSobineq}) holds and $h \in C^{\infty}_0 (\Omega)$, then 
\begin{equation*}
\begin{aligned}
 \Vert h\Vert_{L^{2} (\Omega, q\, dx)} & = \alpha \sup_{\{ \phi: \Vert \phi \Vert_{L^2(\Omega, q \, dx)} \leq 1/\alpha\} } \left|\int_{\Omega} h \phi q \, dx  \right| \\
& \leq  \alpha \sup_{\{ \phi: \Vert \phi q \Vert_{L^{-1,2}(\Omega)} \leq 1\} } \left|\int_{\Omega} h \phi q \, dx  \right|   \\ 
& \leq \alpha  \sup_{\{ \psi: \Vert \psi  \Vert_{L^{-1,2}(\Omega)} \leq 1\} } \left|\int_{\Omega} h \psi \, dx  \right| = \alpha \Vert h \Vert_{L_0^{1,2} (\Omega)} .
\end{aligned}
\end{equation*} 

Since $G(x,y)$ is symmetric, $T$ is a self-adjoint operator on $L^2 (\Omega, q \, dx)$, and hence 
\begin{equation*}
 ||T||_{L^2 (\Omega, q \, dx) \rightarrow L^2 (\Omega, q\, dx)} = \sup_{\{ f:\Vert f \Vert_{L^2 (\Omega, q\, dx)} \leq 1 \}} \left| \langle Tf, f \rangle_{L^2 (\Omega, q\,dx)} \right| .
 \end{equation*}
In computing this supremum, we can assume $f \in \mathcal{B} = \{ f \in C_0^{\infty} (\Omega) : \Vert f \Vert_{L^2 (q \, dx)} \leq 1 \} $.  For $f \in \mathcal{B}$, we have that $fq \in L^1 (\Omega)$.  Hence we obtain
%\begin{equation}\label{eqnfornormT} 
\begin{equation}\label{eqnfornormT} 
\begin{aligned}
  ||T||_{L^2 (\Omega, q\, dx) \rightarrow L^2 (\Omega, q\, dx)}  & = 
\sup_{f \in \mathcal{B}} \left| \langle G(fq), fq \rangle_{L^2 (\Omega, \, dx)} \right| \\ & =  \sup_{f \in \mathcal{B}}  \Vert fq \Vert^2_{L^{-1, 2} (\Omega) } , 
\end{aligned}
\end{equation}
by (\ref{dual-Sob11}).  

Suppose (\ref{impliesnormcond2}) holds for all $h \in C_0^{\infty} (\Omega)$.  Since (\ref{secSobineq}) implies (\ref{firstSobineq}), we have 
\[ ||T||_{L^2 (\Omega, q\, dx) \rightarrow L^2 (\Omega, q\, dx)} =  \sup_{f \in \mathcal{B}}  \Vert fq \Vert_{L^{-1, 2} (\Omega) } \leq \beta^2.   \]
Conversely, if $T$ is bounded on $L^2 (q\, dx)$, then by (\ref{eqnfornormT}), we have 
\[
 \Vert f q \Vert^2_{L^{-1,2} (\Omega)} \leq \Vert T \Vert \Vert f \Vert^2_{L^2 (\Omega, q \, dx)},
 \] 
 first for all $f \in C^{\infty}_0 (\Omega)$, but then as a consequence of density, for all $f \in L^2 (\Omega, q \, dx)$. Since (\ref{firstSobineq}) implies (\ref{secSobineq}), we obtain $\Vert h \Vert^2_{L^2 (\Omega, q \, dx) } \leq \Vert T \Vert \Vert h \Vert^2_{L^{1,2} (\Omega)}$ for all $h \in C^{\infty}_0 (\Omega)$.  Hence $\beta^2 \leq \Vert T \Vert_{L^2 (\Omega, q \, dx) \rightarrow L^2 (\Omega, q \, dx)}$.
\end{proof}

The next step utilizes estimates from \cite{FNV}.  In that paper, a general $\sigma$-finite measure space $(X, \, d\omega)$ and integral  operator $T$ defined by $Tf(x) = \int_X K(x,y) f(y) \, d \omega (y)$ are considered.  Here $K : X \times X \rightarrow (0, \infty]$ is symmetric and quasi-metrically modifiable, which means that there exists a measurable function $m: X \rightarrow (0, \infty)$ (the ``modifier"), such that for $\tilde{K} (x,y) = \frac{K(x,y)}{m(x) m(y)}$ we have that $d (x,y) = 1/ \tilde{K}(x,y)$ satisfies the quasi-metric condition
\[ d(x,y) \leq \kappa (d(x,z) + d(z,y))  \]
for some constant $\kappa >0$ and all $x,y, z \in X$.  For $j \geq 2$, we define $K_j (x,y) = \int_X K_{j-1} (x,z) K(z,y) \, d\omega (z)$.  Then the $j^{th}$ iterate $T^j$ of $T$ has the form $T^j f(x) = \int_X K_{j} (x,y) f(y)  \, d \omega (y)$.  The formal solution to the equation $v = Tv + m$ is 
\[
v_0 = m + \sum_{j=1}^{\infty} T^j m , \] 
for a modifier $m$.  Then Corollary 3.5 in \cite{FNV} states that there exists $c>0$ depending only on $\kappa$ such that 
\begin{equation}\label{newu0lowest}
m e^{c (Tm)/m} \leq v_0,   
\end{equation}
and, if in addition $\Vert T \Vert_{L^2 (\omega) \rightarrow L^2 (\omega)} <1$, then there exists a constant $C >0$ depending only on $\kappa$ and $\Vert T \Vert$ such that 
\begin{equation}\label{newu0uppest}
 v_0 \leq m e^{C (Tm)/m} .
\end{equation}

To apply this result to our case, we let $X = \Omega, d \omega = q(y) \, dy$ and $K(x,y) = G(x,y)$. Note that 
\eqref{operator-tf} holds for $T$ defined on $X$ as above.  As noted in \cite{FV}, p. 118, or \cite{FNV}, p. 905, the equivalence \eqref{green-smootha} in the case 
$n \ge 3$ combined with 
\eqref{G1est} 
 shows that $K$ is quasi-metrically modifiable with modifier 
 $m(x) = \delta (x) = \text{dist}(x, \partial \Omega)$. (We take this opportunity to note a misprint in \cite{FNV}: the power of $|x-y| + \delta(x) + \delta (y)$ in equation (1.6) should be $\alpha$, not $\alpha/2$; this error has no bearing on the validity of the results in that paper.)  For $n=2$, it remains true that $m = \delta$ is a modifier for $K$; this fact follows from estimates \eqref{green-smoothb} and 
 \eqref{G1est}  (see \cite{H}, Proposition 8.6 and Corollary 9.6). Then by (\ref{defngreenpot}) and (\ref{defGj}), we have $K_j (x,y) = G_j (x,y)$ for all $j \geq 1$.  Hence 
\[ T^j( G1) (x) = \int_{\Omega} G_j (x,y) q(y) \int_{\Omega} G(y,z) \, dz dy \]
\[ = \int_{\Omega} \int_{\Omega} G_j (x,y) G(y,z) q(y) \, dy \, dz = G_{j+1} 1 (x),  \]
where $G_{j}$ is the integral operator defined by $G_{j} f(x)= \int_{\Omega} G_{j} (y) f(y) \, dy$.  Hence 
\begin{equation}\label{Neumannseries}
  u_0 = \mathcal{G} 1 = \sum_{j=1}^{\infty} G_j 1 = G1 + \sum_{j=1}^{\infty} G_{j+1} 1 =  G1 + \sum_{j=1}^{\infty} T^j (G1).   
\end{equation}
However, we noted in (\ref{G1est}), $G1$ and $\delta$ are pointwise equivalent.  Hence $u_0 \approx \delta + \sum_{j=1}^{\infty} T^j \delta = v_0$.  Therefore by (\ref{newu0lowest}), there exist constants $c_1>0$ and $c>0$ such that
\begin{equation}\label{estlowu0}
 u_0 \geq c_1 \delta e^{c G(q\delta)/\delta},
\end{equation}
and, if we assume $\Vert T \Vert_{L^2 (\Omega, q \, dx) \rightarrow L^2 (\Omega, q \, dx)} < 1$, then (\ref{newu0uppest}) gives the estimate
\begin{equation} \label{estu0}
 u_0 \leq C_1 e^{C G(q \delta)/\delta} ,
\end{equation}
for some constants $C_1 >0$ and $C>0$.  

\begin{proofof}{Theorem \ref{u0thm}}.  First suppose (\ref{impliesnormcond2}) holds for some $\beta\in (0,1)$.   We note that we then have the inequality 
\begin{equation} \label{Mazyacond}
  \int_\Omega h^2  q \, dx \le \beta^2 ||h||^2_{L^{1, 2}_0(\Omega)}  
\end{equation} 
for all $h \in L^{1, 2}_0(\Omega)$, by an approximation argument, as follows.  Let $h_n \in C^{\infty}_0(\Omega)$ be a sequence of functions converging to $h$ in ${L^{1, 2}_0(\Omega)}$.  Then by the Sobolev imbedding theorem, $h_n$ converges to $h$ in $L^{p^*}$ for some $p^{*}\geq 1$, so by passing to a subsequence we can assume $h_n$ converges to $h$ a.e.  Because of (\ref{impliesnormcond2}), $h_n$ is Cauchy in $L^2 (\Omega, q \, dx)$ and hence converges in $L^2 (\Omega, q \, dx)$ to some function $\tilde{h}$.  Since there is a subsequence of $h_n$ converging $q dx$-a.e. to $\tilde{h}$, we must have $\tilde{h} = h$ a.e. with respect to $q dx$.  Hence we can let $n \rightarrow \infty$ in  $\int_\Omega h_n^2  q \, dx \le \beta^2 \int_{\Omega} |\nabla h_n|^2\, dx$ to obtain (\ref{Mazyacond}). 

Observe that $G1 \in C(\overline{\Omega})$, $G1 =0$ on $\partial \Omega$, and, by (\ref{dual-Sob11}) and (\ref{G1est}), 
\[  \int_{\Omega} |\nabla G1 |^2 \, dx = \int_{\Omega} G1 \, dx \leq C \int_{\Omega} \delta (x) \, dx < \infty .  \]
Hence $G1 \in L^{1,2}_0(\Omega)$.  By the remark in the last paragraph, $G1 \in L^2(\Omega, q \, dx)$.  Since $G1 \approx \delta$, this means that $\delta q \in L^1 (\Omega, \delta \, dx)$.  By Lemma \ref{lemma-weak} (i), $G(\delta q) \in L^1 (\Omega, \, dx)$.  In particular, $G (\delta q) < \infty $ a.e.  

By our assumption (\ref{impliesnormcond2}) and Lemma \ref{leastconst}, the operator $T$ defined by (\ref{operator-tf}) has $\Vert T \Vert_{L^2 (\Omega, q \, dx) \rightarrow L^2 (\Omega, q \, dx)} \leq \beta^2 <1$.  Hence by (\ref{estu0}), $u_0=  \mathcal{G} 1 $ satisfies (\ref{u0upperestimate}) and $u_0  < \infty $ a.e. By Lemma \ref{finiteimpliessoln}, $u_0 \in L^1 (\Omega, \, dx) \cap L^1 (\Omega, \, \delta q \, dx)$, and $u_0$ is a positive very weak solution of (\ref{schro-1}).  

Since $\Vert T \Vert_{L^2 (\Omega, q \, dx) \rightarrow L^2 (\Omega, q \, dx)} <1$, the operator $(I-T)^{-1} = \sum_{j=0}^{\infty}T^j$ is bounded on $L^2 (q \, dx)$.  Hence $u_0 \in L^2 (q, \, dx)$, by (\ref{Neumannseries}).  Since $u_0 = G(u_0 q + 1)$, we have 
 \begin{equation*}
\begin{aligned}
\int_\Omega |\nabla u_0|^2 dx & = \int_\Omega |\nabla G(u_0 q +1)|^2 dx \\
& = \int_\Omega G(u_0 q +1) (u_0 q+1) \, dx  = \int_\Omega (u_0^2 q+u_0) \, dx . 
\end{aligned}
 \end{equation*}
by (\ref{dual-Sob11}).  Since $u_0 = G(u_0 q +1)$ is $0$ on $\partial \Omega$, we obtain $u_0 \in L^{1,2}_0 (\Omega)$.  

We now show that $u_0$ is a weak solution of (\ref{u0equation}).  Since $u_0 \in L^{1,2}_0 (\Omega)$, we must show that 
\[  \int_{\Omega} \nabla u_0 \cdot \nabla h \, dx = \int_{\Omega} hu_0 q + h \, dx , \]
for all $h \in L^{1,2}_0 (\Omega)$.  Let $h_n$ be a sequence in $C_0^{\infty} (\Omega)$ converging to $h$ in the norm on $L^{1,2}_0 (\Omega)$.  Then
\[ \int_{\Omega} \nabla u_0 \cdot \nabla h_n \, dx = - \int_{\Omega} u_0  \triangle h_n \, dx   = \int_{\Omega} h_n u_0 q + h_n \, dx  \]
because $u_0$ is a very weak solution of (\ref{u0equation}).  The left side converges as $n \rightarrow \infty$ to $ \int_{\Omega} \nabla u_0 \cdot \nabla h \, dx$, because $h_n$ converges to $h$ in $L^{1,2}_0 (\Omega)$.  By (\ref{Mazyacond}), which we now know is valid for all $h$ in $L^{1,2}_0 (\Omega)$, we have that $h_n$ converges to $h$ in $L^2 (\Omega, q \, dx)$.  We also know that $u_0 \in L^2 (\Omega, q \, dx)$.  Hence using the Cauchy-Schwarz inequality in $L^2 (\Omega, q \, dx)$ we see that $\int_{\Omega} h_n u_0 q \, dx$ converges to $\int_{\Omega} h u_0 q \,dx $.  The imbedding of $L^{1,2}_0 (\Omega, \, dx)$ in $L^1 (\Omega, dx)$ shows that $\int_{\Omega}  h_n \, dx$ converges to $\int_{\Omega} h  \, dx$.  Therefore $u_0$ is a weak solution of (\ref{u0equation}).  

Now suppose $u \in L^1 (\Omega, \, dx) \cap L^1 (\Omega, \, \delta q \, dx)$ and $u$ is a positive very weak solution of (\ref{u0equation}).  By Definition \ref{defveryweaksoln} and Lemma \ref{integeqnequiv}, $u$ satisfies the integral equation $u= G(qu)+ G1 = Tu + G1$ a.e., for $T$ defined by (\ref{operator-tf}).  Since $G1 \geq 0$, we have $T(u) \leq u$ a.e., with $u \geq G1 >0$ and $u < \infty$ a.e.  Hence by Schur's test for integral operators, we have $\Vert T \Vert_{L^2 (\Omega, q \, dx) \rightarrow L^2 (\Omega, q \, dx)} \leq 1$.  By Lemma \ref{leastconst}, it follows that (\ref{impliesnormcond2}) holds with $\beta =1$.  Since $u_0= \mathcal{G} 1 $ is the minimal positive very weak solution of (\ref{u0equation}), we have $u_0 \leq u$, hence  (\ref{u0lowerestimate}) holds because of (\ref{estlowu0}).
\end{proofof}

We turn now to equation (\ref{dirichlet}).  By Lemma \ref{finiteimpliessoln2}, the essential issue is whether $u_1 = 1 + \mathcal{G} (q)$ is finite a.e., 
or equivalently $u_1 \in L^1 (\Omega)$.  We will use the relation between $u_0$ and $u_1$ exhibited by the following simple lemma.

\begin{Lemma}\label{u1finite}
Let $\Omega, q, u_0$, and $u_1$ be as in Theorems \ref{u0thm} and \ref{nonlinthm}.  Then $u_1 \in L^1 (\Omega, dx)$ if and only if $u_0 \in L^1 (\Omega, q \, dx)$.
\end{Lemma}

\begin{proof}
Since $u_0 = \mathcal{G}1$, Fubini's theorem and the symmetry of $\mathcal{G} (x,y)$ yield
\begin{equation}\label{u0u1int}
 \int_{\Omega} u_1 \, dx = \int_{\Omega} 1 \, dx + \int_{\Omega} \int_{\Omega} \mathcal{G}(x,y) q(y) \, dy \, dx = |\Omega| + \int_{\Omega} u_0  q \, dy .  
\end{equation} 
\end{proof}

The following convergence lemma will be useful.  

\begin{Lemma}\label{convlemma}
Suppose $\Omega \subset \R^n$ is a bounded $C^2$ domain, for $n \geq 2$.  Suppose $q \in L^1 (\Omega, dx)$ and $q$ has compact support in $\Omega$.  Suppose $\phi \in L^1 (\Omega, q \, dx)$.  Let $z \in \partial \Omega$, and let $\{x_j \}_{j=1}^{\infty}$ be a sequence in $\Omega$ converging to $z$ in the normal direction.  Then
\[ \lim_{j \rightarrow \infty} \frac{G(\phi q)(x_j)}{\delta (x_j)} = \int_{\Omega} P(y,z) \phi (y) q(y) \, dy . \]
\end{Lemma}

\begin{proof}  Recall that $P(y,z)$ is the normal derivative of $G(y,x)$ as $x \rightarrow z$, $x \in \Omega$.  Hence $\lim_{j \rightarrow \infty} G(y, x_j) / \delta (x_j) = P(y,z)$.  Since $G$ is symmetric,
\[ \lim_{j \rightarrow \infty} \frac{G(\phi q)(x_j)}{\delta (x_j)} = \lim_{j \rightarrow \infty} \int_{\Omega} \frac{G(y, x_j)}{\delta (x_j)} \phi (y) q(y) \, dy .   \]
There is some constant $c_1>0$ such that $|y-x_j| \geq c_1 $ for all $y$ belonging to the support of $q$ and all sufficiently large $j$.  Hence (\ref{green-smootha}) shows that $G(y, x_j)/\delta (x_j)$ is bounded for all large enough $j$.  The result follows by the dominated convergence theorem.
\end{proof}

We will need an elementary lemma on quasi-metric spaces due to Hansen and Netuka (\cite{HaN}, Proposition 8.1 and Corollary 8.2); in the context of normed spaces it was proved 
earlier by Pinchover \cite{P}, Lemma A.1. 

\begin{Lemma}\label{hansen} Suppose $d$  is a quasi-metric on a set $\Omega$ with quasi-metric constant $\kappa$. Suppose $z \in X$. Then 
\begin{equation}\label{quasi-cond} 
\tilde d(x,y) = \frac{d(x,y)}{d(x,z) \cdot d(y,z)}, \qquad x, y 
\in \Omega \setminus\{z\},
\end{equation} 
is a quasi-metric on $\Omega\setminus\{z\}$ with quasi-metric constant 
$4 \kappa^2$. 
\end{Lemma}

\begin{proofof}{Theorem \ref{nonlinthm}}. First suppose (\ref{impliesnormcond2}) holds for some $\beta \in (0,1)$ and (\ref{balayagecond}) holds for the constant $C$ in (\ref{u0upperestimate}).  By Theorem \ref{u0thm}, $u_0 \in  L^1 (\Omega, dx)\cap L^1 ( \Omega, \delta q \, dx)$ and $u_0$ satisfies (\ref{u0upperestimate}).   By Corollary \ref{ident} and Lemma \ref{finiteimpliessoln}, $u_0 = G( u_0 q +1)$ at every point of $\Omega$.

Let $z \in \partial \Omega$.  We claim that
\begin{equation} \label{Poissonineq}
\int_{\Omega} P(y,z) u_0 (y) q(y) \, dy \leq C_1 e^{C P^* (\delta q) (z)},
\end{equation} 
where $C_1$ and $C$ are the constants from (\ref{u0upperestimate}).  

We first prove this claim under the additional assumption that $q$ is compactly supported in $\Omega$, so $q \in L^1 (\Omega, \, dx)$.  Since $\delta$ is bounded above and below, away from $0$, on the support of $q$, the condition $u_0 \in L^1 (\Omega, \delta q)$ is equivalent to the condition $u_0 \in L^1 (\Omega, q\, dx)$.  Let $\{x_j \}_{j=1}^{\infty}$ be a sequence in $\Omega$ converging to $z$ in the normal direction. Applying Lemma \ref{convlemma} with $\phi = u_0$, we obtain 
\[ \int_{\Omega} P(y,z) u_0 (y) q(y) \, dy  = \lim_{j \rightarrow \infty} \frac{G(u_0 q) (x_j)}{\delta (x_j)} . \]
Applying the equation $u_0 = G(u_0 q +1)$ and (\ref{u0upperestimate}), 
\[ \frac{G(u_0 q) (x_j)}{\delta (x_j)} \leq \frac{u_0 (x_j)}{\delta (x_j)} \leq C_1 e^{C G(\delta q)(x_j)/\delta (x_j)}.  \]
Taking the limit and applying Lemma \ref{convlemma} with $\phi = \delta \in L^1 (\Omega, q \, dx)$ gives 
\[ \int_{\Omega} P(y,z) u_0 (y) q(y) \, dy \leq C_1 e^{C \int_{\Omega} P(y,z) \delta (y) q(y) \, dy}   = C_1 e^{C P^* (\delta q) (z)} .\]

We now remove the assumption that $q$ is compactly supported in $\Omega$.  Let $q \in L^1_{loc} (\Omega)$.  Let $\{\Omega_k \}_{k=1}^{\infty}$ be an exhaustion of $\Omega$ by smooth subdomains with compact closure such that $\overline{{\Omega}}_{k+1} \subset \Omega_k$, $k=1,2, \ldots$.  Define $q_k = q \chi_{\Omega_k}$.  Then each $q_k$ has compact support in $\Omega$.  Define the iterated Green's kernels $G_j^{(k)} (x,y)$ for $j=1, 2, 3, \dots$, and $\mathcal{G}^{(k)}$, as in (\ref{defngreenpot}), (\ref{defGj}), and (\ref{defGreenfcnSchr}), except with $q$ replaced by $q_k$.  Let $u_k = \mathcal{G}^{(k)} 1$.  By repeated use of the monotone convergence theorem, $G_j^{(k)} (x,y)$ increases monotonically as $k \rightarrow \infty$ to $G_j (x,y)$ for each $j$, $\mathcal{G}^{(k)}(x,y)$ increases monotonically to $\mathcal{G} (x,y)$, and $u_k$ increases monotonically to $u_0$.  Applying the compact support case gives  
\[ \int_{\Omega} P(y,z) u_k (y) q_k(y) \, dy \leq C_1 e^{C P^* (\delta q_k) (z)} \leq  C_1 e^{C P^* (\delta q) (z)}  . \]
Then the monotone convergence theorem yields (\ref{Poissonineq}).

We integrate  (\ref{Poissonineq}) over $\partial \Omega$:
\[ \int_{\partial \Omega} \int_{\Omega}   P(y,z) u_0 (y) q(y) \, dy d\sigma (z) \leq   C_1\int_{\partial \Omega}  e^{C P^* (\delta q) (z)} \, d \sigma (z).   \]
By Fubini's theorem and the fact that $P(1) =1$, the left side is just $\int_{\Omega} u_0 q \, dx$.  Hence (\ref{u0upperestimate}) implies $u_0 \in L^1 (\Omega, q \, dx)$.  By Lemma \ref{u1finite} and (\ref{u0u1int}), we obtain $u_1 \in L^1 (\Omega, \, dx)$ and (\ref{u1L1ineq}) holds.  Hence $u_1 < \infty$ a.e.  Then Lemma \ref{finiteimpliessoln2} shows that $u_1$ is a very weak solution of (\ref{dirichlet}).  

Next we prove the pointwise estimate (\ref{pointwiseestu1}).  Since for all $x \in \Omega$ we have $\int_{\partial \Omega} P(x,z) \, d \sigma (z) =1$, it follows  
\begin{equation*}
\begin{aligned}
 u_1 (x) & = 1 + \mathcal{G} q (x) = 1 + \sum_{j=1}^{\infty} \int_{\Omega} G_j (x,y) q(y) \, dy \\ & = \int_{\partial \Omega} \left ( P(x,z) + \sum_{j=1}^{\infty} \int_{\Omega} G_j (x,y)P(y,z) q(y) \, dy  \right )d \sigma (z) \\ 
&  
= \int_{\partial \Omega} \sum_{j=0}^{\infty} T^j ( P(\cdot, z)) (x) \, d \sigma (z).  
\end{aligned}
\end{equation*}

The following estimates of the Poisson kernel are well known (see \cite{CZ}): there exist constants $c=c(\Omega)$, $C=C(\Omega)$ so that, for $x\in  \Omega$ and $ z \in \partial \Omega$:  
\begin{equation}\label{poisson-2} 
\frac{c \,\delta(x)}{|x-z|^n} \le P(x,z)\le \frac{C \, \delta(x)}{|x-z|^n}. 
\end{equation} 

Fix $z \in \partial \Omega$ for the moment.  We claim that $m(x) = P(x,z)$ is a modifier for $K(x,y)=G(x,y)$.  To see this fact, define a quasi-metric $d$ on $\overline{\Omega}$, for $n \ge 3$, by: 
 \[  d(x,y) = |x-y|^{n-2} \, [ |x-y|^2 + \delta(x)^2 + \delta(y)^2 ], \quad  x, y \in \overline{\Omega}.  \]
Notice that, for $z \in \partial \Omega$, we have  $d(x,z) \approx |x-z|^n$ since $|x-z|\ge \delta(x)$ and $\delta(z)=0$. Hence by (\ref{poisson-2}),
\[ m(x) \approx \delta(x)  / d(x, z),  \qquad x \in \Omega. \] 
Using  (\ref{green-smootha}) together with the preceding inequalities, we estimate the modified kernel $\tilde{K}$:  
 \begin{equation}\label{eq-mod} 
  \tilde{K}(x,y) = \frac{G(x,y)}{m(x) \cdot m(y)} \approx  \frac {d(x,z) \cdot d(y,z)}{d(x,y)}. 
  \end{equation} 
By Lemma~\ref{hansen}, $\tilde{K}$ is a quasi-metric kernel on  $\overline{\Omega}\setminus\{z\}$, and hence on  $\Omega$. Notice that all the constants of equivalence depend only on $\Omega$, but not on $z$. 

Similarly, for $n=2$, we  invoke  (\ref{green-smoothb}) to define a quasi-metric on $\overline{\Omega}$ 
using an extension  by continuity of the quasi-metric originally defined on $\Omega$ by 
 \[  d(x,y) = \delta(x) \delta(y) \left [\ln \left(1+ \frac{\delta(x) \delta(y)}{|x-y|^2} \right) \right]^{-1}, \quad  x, y \in \Omega.  \] 
 In other words, for $x \in \Omega$ and $z \in \partial \Omega$, we set  
  \begin{equation*}
  \begin{aligned}
 d(x, z) & = \lim_{y \to z, \, y \in \Omega} d(x, y) \\& = \lim_{y \to z, \, \delta(y)\to 0} \delta(x) \delta(y) \left [\ln \left(1+ \frac{\delta(x) \delta(y)}{|x-y|^2} \right) \right]^{-1}
 = |x-z|^2.  
   \end{aligned}
     \end{equation*}
 The same formula will be used if both $x, z \in \partial \Omega$, so that 
 \[  d(x, z) = |x-z|^2 \quad  \mbox{for all}  \, \, \, x \in \overline{\Omega}, \, \, z \in \partial{\Omega}.   \]
 Clearly, the extended function $d$ satisfies the quasi-triangle inequality on $\overline{\Omega}$. Moreover, 
for $z \in \partial \Omega$, we have  
by (\ref{poisson-2}),
\[ m(x) =P(x, z)\approx \delta(x)  / d(x, z)   \quad  \mbox{for all}  \, \, \, x  \in \Omega . \] 
By Lemma~\ref{hansen} the modified kernel $\tilde{K}(x, y)=\frac{G(x,y)}{m(x)m(y)}$  is a quasi-metric kernel on  $\Omega$, since it satisfies  (\ref{eq-mod}) as in the case 
$n \ge 3$.

Applying (\ref{newu0lowest}) and (\ref{newu0uppest}) to estimate $\sum_{j=0}^{\infty} T^j m$, we obtain: 
\[  c_1 \, P(x,z)  \,  e^{c_2 \int_\Omega G(x,y) \frac{P(y,z)}{P(x,z)} q(y) \, dy}  \le  \sum_{j=0}^{\infty} T^j ( P(\cdot, z)) (x)  \]
\[ \leq C_2 \, P(x,z) \,  e^{C_3 \int_\Omega G(x,y) \frac{P(y,z)}{P(x,z)} q(y) \, dy},   \]
where the constants do not depend on $x\in \Omega$ and  $z\in \partial \Omega$.   Substituting into the expression for $u_1$ above, we obtain (\ref{pointwiseestu1}) as well as the lower estimate
\begin{equation}\label{u1ptwiselow}
 u_1 (x) \geq c_1 \int_{\partial \Omega} e^{c_2 \int_{\Omega} G(x,y) \frac{P(y,z)}{P(x,z)} q(y) \, dy    } P(x,z) \, d \sigma (z) .
\end{equation}

For the converse, suppose $u$ is a positive very weak solution of (\ref{dirichlet}).  By the remarks after Definition \ref{defveryweak2}, $u$ satisfies the integral equation $u = 1 + G(qu) = 1 + Tu$, for $T$ defined by (\ref{operator-tf}).  Hence $0 <u< \infty$ a.e., and $Tu \leq u$.  By Schur's test, $\Vert T \Vert_{L^2 (\Omega, q \, dx) \rightarrow L^2 (\Omega, q \, dx)} \leq 1$.  By Lemma \ref{leastconst}, inequality (\ref{impliesnormcond2}) holds with $\beta =1$.

Since $u_1 = 1 + \mathcal{G} q$ is the minimal positive very weak solution of (\ref{dirichlet}), we have $\mathcal{G} q < u_1 \leq u$, hence $\mathcal{G}q < \infty$ a.e.  By Lemma \ref{lemma-weak}, $q \in L^1 (\Omega, \delta \, dx)$ and by Lemma \ref{finiteimpliessoln}, $\mathcal{G}q \in L^1 (\Omega, \, dx)$.  Hence $u_1 = 1 + \mathcal{G} q \in L^1 (\Omega, dx)$.  By Lemma \ref{u1finite}, $u_0 \in L^1 (\Omega, q \, dx)$.  

Let $z \in \partial \Omega$.  We claim that there exists $c_2>0$ depending only on $\Omega$ such that 
\begin{equation} \label{Poissonineq2}
c_1 e^{c P^* (\delta q) (z)} \leq \int_{\Omega} P(y,z) u_0 (y) q(y) \, dy + c_2 ,
\end{equation} 
where $c_1$ and $c$ are the constants from (\ref{u0lowerestimate}).  By the same exhaustion process that was used in the forward direction, it is sufficient to prove (\ref{Poissonineq2}) under the assumption that $q$ has compact support in $\Omega$.  Under that assumption, we have that $\delta $ is bounded above, and below away from $0$, on the support of $q$, so $\delta q \in L^1 (\Omega, dx)$.  Let $\{ x_j \}_{j=1}^{\infty}$ be a sequence of points in $\Omega$ which converge to $z$ in the normal direction. By Lemma \ref{convlemma} with $\phi=\delta$, 
\[ c_1 e^{c P^* (\delta q) (z) } = e^{c \int_{\Omega} P(y,z) \delta (y) q(y) \, dy} = \lim_{j \rightarrow \infty}c_1 e^{c G(\delta q) (x_j) / \delta (x_j) } . \]
By Theorem \ref{u0thm}, $u_0$ satisfies the estimate in (\ref{u0lowerestimate}). Hence
\[ c_1 e^{c G(\delta q) (x_j) / \delta (x_j) } \leq \frac{u_0 (x_j)}{\delta (x_j)}  =\frac{G(u_0 q)(x_j)}{\delta (x_j)} +  \frac{G1 (x_j)}{\delta (x_j)} \leq \frac{G(u_0 q)(x_j)}{\delta (x_j)} + c_2, \]
by (\ref{G1est}).  Because $u_0 \in L^1 (\Omega, q \, dx)$, taking the limit and applying Lemma \ref{convlemma} with $\phi = u_0$ gives 
(\ref{Poissonineq2}).  

Integrating (\ref{Poissonineq2}) over $\partial \Omega$, applying Fubini's theorem, and using the fact that $P1=1$, we obtain
\begin{equation*}
\begin{aligned}
 c_1 \int_{\partial \Omega}  e^{c P^* (\delta q) (z)} \, d \sigma(z) & \leq \int_{\partial \Omega} \left(  \int_{\Omega} P(y,z) u_0 (y) q(y) \, dy + c_2  \right) \, d\sigma (z) \\  & = \int_{\Omega} u_0 q \, dx + c_2 |\partial \Omega|  < \infty , 
 \end{aligned}
 \end{equation*}
since $u_0 \in L^1 (\Omega, q\, dx)$.  By Lemma \ref{u1finite} and the minimality of $u_1$, we have $\int_{\Omega} u_0 q \, dx \leq \int_{\Omega} u_1 \, dx \leq \int_{\Omega} u \, dx$, which establishes (\ref{u1L1ineqconv}).  

Now (\ref{pointwiseestu}) follows from (\ref{u1ptwiselow}), since $u \geq u_1$.
\end{proofof}

\begin{proofof}{Corollary \ref{carl}}.
By the John-Nirenberg theorem, $e^{\beta P^* (\delta q)}$ is integrable on $\partial \Omega$, for $\beta$ less than a multiple of the reciprocal of the BMO norm of $P^* (\delta q)$.  Hence if (A) holds for $\epsilon_1$ small enough, then $\ds  \int_{\partial \Omega} e^{C P^* (\delta q)} \, d\sigma  < \infty$, and the conclusions follow from Theorem \ref{nonlinthm}.  

By a standard theorem (see e.g., \cite{PV}, \cite{G}, p. 229), $P^* (\delta q) \in $ BMO$(\partial \Omega)$ with BMO norm bounded by a multiple of the Carleson norm of $\delta q \, dx$.   Therefore (B) for $\epsilon_2$ sufficiently small implies (A).
\end{proofof}

Condition (B) above actually yields (A) with every $\chi_E \, \delta q$ in place of $\delta q$, for any measurable $E \subset \Omega$, and the converse is also true  
(see \cite{PV}). 

We now turn to the case of $\Omega = \R^n$.  For $0< \alpha <n$, let $I_{\alpha}= (-\triangle)^{-\alpha/2}$ denote the Riesz potential defined by
\[ I_\alpha f (x) = c_{n, \alpha} \int_{\R^n} |x-y|^{\alpha-n} f(y) \, dy,    \]
for some constant $c_{n, \alpha} >0$.  If $f \ge 0$, then the 
Riesz potential $I_\alpha f (x)$ is finite a.e. in $\R^n$ if and only if 
\begin{equation}\label{riesz-cond}
 \int_{\R^n} \frac{f(y) \, dy}{(|y|+1)^{n-\alpha}} < +\infty.
 \end{equation}
Otherwise  $I_\alpha f \equiv +\infty$ on $\R^n$ (\cite{L}, Sec. I.3).  If (\ref{riesz-cond}) holds, then 
\[ \lim \inf_{x \rightarrow \infty} I_{\alpha} f (x) =0 .  \]
The kernel of $I_2$ is the Green's function $G(x,y)= c_{n}  |x-y|^{2-n}$ and the Green's operator $G$ coincides with $I_2$.  Define the iterates $G_j$ and $\mathcal{G}$ by (\ref{defGj}) and (\ref{defGreenfcnSchr}).   

We consider positive solutions $u$ to the  Schr\"{o}dinger equation 
 \begin{equation}\label{riesz-eq}
  -\triangle u  = q u + f \quad \text{in} \, \, \Omega,
 \end{equation}
where $q \ge 0$ is a given non-negative potential and $f\ge 0$ is a function such that 
\begin{equation}\label{condfonRn}
  \int_{\R^n} \frac{f(y) \, dy}{ (1+|y|)^{n-2}} < +\infty. 
\end{equation}
Equation (\ref{riesz-eq}) is understood in the distributional sense. Equivalently (see \cite{L}, Sec. I.5), $u \in L^1_{loc}(\R^n)$, $u \ge 0$, is a solution to (\ref{riesz-eq}) if 
\[  \int_{\R^n} \frac{u(y) \, q(y) \, dy}{ (1+|y|)^{n-2} } < +\infty,\]
and 
  \begin{equation}\label{riesz-inteq}
 u = I_2 (qu) + I_2 f + c \quad \text{a.e.} , 
  \end{equation}
where $c$ is a non-negative constant and $\lim \inf_{x \rightarrow \infty} u(x) = c$.  

Since $f=1$ does not satisfy (\ref{condfonRn}), we do not obtain conditions for the solvability of (\ref{u0equation}) on $\R^n$.  On a bounded domain, the results for (\ref{u0equation}) in Theorem \ref{u0thm} were used to obtain our results in Theorem \ref{nonlinthm} for (\ref{dirichlet}).  Nevertheless we obtain results for (\ref{dirichlet}) on $\R^n$.  

We first note that Lemma \ref{leastconst} holds for $\Omega = \R^n$.  Define the operator $T$ by (\ref{operator-tf}) with $\Omega = \R^n$.  Define the homogeneous Sobolev space $L^{1, 2}_0(\R^n)$ to be the closure of  $C^\infty_0(\R^n)$ with respect to the norm $||f||_{L^{1, 2} (\R^n)} = ||(-\triangle)^{1/2} f||_{L^2(\R^n)}$.  The dual of $L^{1, 2}_0(\R^n)$ is isometrically isomorphic to $L^{-1, 2}(\R^n)$ defined via the norm $||f||_{L^{-1, 2} (\R^n)} = ||(-\triangle)^{-1/2} f||_{L^2(\R^n)}$.  Because of the semi-group property $I_{\alpha/2} * I_{\alpha/2}=I_\alpha$ of the Riesz kernels, we have, for $f \ge 0$ (or if $f$ is a finite signed measure),  
\[ ||f||^2_{L^{-1, 2} (\R^n)} = \int_{\R^n} (I_1 f)^2 \, dx = \int_{\R^n} f I_2 f \, dx = \int_{\R^n} f Gf \, dx, \]
which is the analogue of (\ref{dual-Sob11}).  With this result, the proof of Lemma \ref{leastconst} carries over to $\R^n$ and we obtain that $||T||_{L^2(\R^n, q \, dx) \to L^2(\R^n, q \, dx)}=\beta^2$, where $\beta$ is the least constant in (\ref{trace-1}).

\begin{proofof}{Theorem \ref{riesz-est}}.
First suppose (\ref{trace-1}) holds for some $\beta \in (0,1)$, and (\ref{finite-int}) holds. Then the equation $-\triangle u = qu$ with $\lim \inf_{x \rightarrow \infty} u(x) = 1$ is equivalent to $u = I_2 (qu) + 1 = T(u) +1$, by (\ref{riesz-inteq}) with $f=0$ and $c=1$.  By the analogue of Lemma \ref{leastconst} for $\R^n$ just noted, the operator $T$ has norm less than 1 on $L^2 (\R^n , q \, dx)$.  Since the Riesz kernel $G(x,y)$ is quasi-metric, Theorem 3.1 in \cite{FNV} (i.e, Corollary 3.5 in \cite{FNV}, or (\ref{newu0uppest}) with $m=1$) states that 
\[
1 + \sum_{j=1}^{\infty} T^j 1 \leq e^{C T1} ,
\]
 where $C$ depends only on $n$ and $\beta$.  Note that 
 \[
 T^j 1 (x) = \int_{\R^n} K_j (x,y)  q(y) \, dy = 
\int_{\R^n} G_j (x,y) q(y) \, dy = G_j q (x) , 
\] 
since $K_j (x,y)= G_j (x,y)$ by (\ref{defngreenpot}) and (\ref{defGj}).  Hence 
\[ u_1 = 1 + \mathcal{G} (q) = 1 + \sum_{j=1}^{\infty} T^j 1 \leq e^{C T1} = e^{C I_2 q},\] 
so (\ref{upper-estRn}) holds.  By (\ref{finite-int}), $I_2 q< \infty$ a.e., so $u_1$ defines a positive solution to $-\triangle u = qu$ with $\lim \inf_{x \rightarrow \infty} u(x) =1$.  

Conversely, suppose $u$ is a nonnegative solution of (\ref{eqnRn}), or equivalently, $u = I_2 (qu) + 1 = Tu + 1$.  Then $1 \leq u < \infty$ a.e., so by Schur's test we have $\Vert T \Vert_{L^2(\R^n, q \, dx) \to L^2(\R^n, q \, dx)}\leq 1$, which, we have seen, implies (\ref{trace-1}) with $\beta =1$.  By iteration of the identity $u=Tu+1$, we see that $u \geq u_1$, so $u_1$ is minimal among positive solutions.  Applying the lower estimate from Theorem 3.1 in \cite{FNV} (i.e., (\ref{newu0lowest}) with $m=1$), we have 
\[ u \geq u_1 \geq e^{cT1} = e^{c I_2 q} ,  \]
where $c$ depends only on $n$, so (\ref{lower-estRn}) holds.  Since $u< \infty$ a.e., we conclude tht $I_2 q < \infty$ a.e., so (\ref{finite-int}) holds.  
\end{proofof}

\bgp 

\section{Nonlinear Equations with Quadratic Growth in the Gradient}\label{riccati}

Let $\Omega \subset \R^n$ be a bounded $C^2$ domain and let $q \in L^1_{loc} (\Omega)$ with $q \geq 0$.  
The definition of very weak solutions of (\ref{nonlineareqn-1}) is consistent with Definitions \ref{defveryweaksoln}, \ref{defveryweak}, and Remark \ref{Remarkmeasure}. A good reference to very weak solutions of elliptic equations  is \cite{MV}. 

\begin{Def}\label{defveryweakricatti}  Let $q\in L^1 (\Omega, \delta \, dx)$.  A function $v \in L^1 (\Omega, dx) $ such that $\int_{\Omega} |\nabla v|^2 \delta \, dx < \infty$ is a very weak solution of (\ref{nonlineareqn-1}) if 
%\label{weakricatti}
\[- \int_{\Omega} v \triangle h \, dx = \int_{\Omega}  |\nabla v|^2 h \, dx + \int_{\Omega} h \, q \, d x , \,\,\, \mbox{for all} \,\,\, h \in C^2_0 (\overline{\Omega}) .\]
\end{Def}

Lemma \ref{integeqnequiv} and Corollary \ref{ident} show that if $v$ is a very weak solution of (\ref{nonlineareqn-1}) then $v$ satisfies the integral equation 
\begin{equation} \label{integralformmeasure}
v = G (|\nabla v|^2 + q) \,\, \hbox{a.e.},
\end{equation}
and if $q \geq 0, v < \infty $ a.e. and $v$ satisfies (\ref{integralformmeasure}), then $v \in L^1 (\Omega, \, dx)$, $\int_{\Omega} |\nabla v|^2 \delta \, dx < \infty$, and $v$ is a very weak solution of (\ref{nonlineareqn-1}).

Corresponding formally to (\ref{nonlineareqn-1}) under the substitution $v= \log u$ is equation 
(\ref{dirichlet}).  However the precise relation between very weak solutions to (\ref{dirichlet}) and (\ref{nonlineareqn-1}) is not as simple as it might appear, as shown by the next example which was first noted by Ferone and Murat in \cite{FM2}. 
  
\begin{Rem} \label{nonequiv} Even for the case $q =0$, there is a very weak solution $v$ of (\ref{nonlineareqn-1}) such that $u=e^v$ is not a very weak solution of (\ref{dirichlet}).  Let $v (x) = \log \left (1+ G(x, x_0)\right)$, where $x_0\in \Omega$ is a fixed pole.  Then standard arguments show that $v$ is a very weak solution of $- \triangle v = |\nabla v|^2$ on $\Omega$ with $v = 0$ on $\partial \Omega$.  However, $u=1+ G(x, x_0)$ satisfies $-\Delta u = \delta_{x_0}$ in $\Omega$, so that $u$ is not a very weak solution of (\ref{dirichlet}). 
\end{Rem}

We will see that if $u_1$ is the minimal positive very weak solution of (\ref{dirichlet}), then $v = \log u_1$ is a very weak solution of (\ref{nonlineareqn-1}).  However, in general, if $v$ is a very weak solution to (\ref{nonlineareqn-1}) then $u=e^v$ is only a supersolution to (\ref{dirichlet}), which is enough to prove Theorem \ref{riccatithm}.

\begin{proofof}{Theorem \ref{riccatithm}}.  First suppose that (\ref{impliesnormcond2}) holds for some $\beta \in (0,1)$, and (\ref{balayagecond}) holds.   By Theorem~\ref{nonlinthm}, the Schr\"odinger equation (\ref{dirichlet}) has a positive very weak solution 
$ u(x)= 1 + \mathcal{G} q$.  (This solution $u$ was called $u_1$ in the statement of the theorem; we call it $u$ in the proof to avoid ambiguity with a sequence $\{ u_k \}_{k=1}^{\infty}$ which will be defined later.) Then $u \in L^1 (\Omega, dx)$ and $u$ satisfies the integral equation $u = 1 + G(q u)$.  Therefore $u:  \Omega \to [1, +\infty]$  is defined everywhere as a positive superharmonic function in $\Omega$ and hence  is quasi-continuous; moreover,   $\text{cap}(\{u=+\infty\})=0$ (see 
\cite{AG}).  By Remark \ref{Remarkmeasure}, 
 $u \in  W^{1,p}_{loc}(\Omega)$ when $p< \frac{n}{n-1}$. We remark that actually, as shown in \cite{JMV},  $u \in W^{1,2}_{loc}(\Omega)$, 
but the proof of this stronger property is somewhat involved, and it will not be used below. 

Define $d \mu = -\triangle u = qu\, dx$, where $ q u \in L^1_{loc}(\Omega)$. 
Let $v=\log u$. Then $0\leq  v < +\infty$-a.e., $v$ is superharmonic in $\Omega$ by Jensen's inequality, and 
 $v  \in W^{1,2}_{loc}(\Omega)$ (see \cite{HKM}, Theorem 7.48; \cite{MZ}, Sec. 2.2). We claim that  
 \begin{equation}\label{ric-eq-1}
- \triangle v = |\nabla v|^2  + q \quad \mbox{in} \, \, \, \, D'(\Omega).
\end{equation}
To prove (\ref{ric-eq-1}), we will apply the integration by parts formula 
 \begin{equation}\label{by-parts}
\int_\Omega g \, d \rho = - \langle g, \Delta r \rangle=  \int_\Omega \nabla g \cdot \nabla r \, dx,
\end{equation}
where $g\in  W^{1,2}(\Omega)$ is compactly supported and quasi-continuous in $\Omega$, and $\rho = -\Delta r$ where $r \in W^{1,2}_{loc}(\Omega)$ is superharmonic (see, e.g.,  \cite{MZ}, Theorem 2.39 and Lemma 2.33).  This proof would simplify if we could apply (\ref{by-parts}) with $g = h/u, \rho = \mu$, and $r=u$, for $h \in C^{\infty}_0 (\Omega)$.  However, we do not have $r \in W^{1,2}_{loc}(\Omega)$, so we require an approximation argument.  For $k \in \N$, let 
\[ u_k = \min (u, \, e^k), \quad v_k=\min (v, \, k), \quad \mbox{and} \quad  \mu_k = - \triangle u_k.\]
Clearly $u_k$ and $v_k$ are superharmonic, hence $\mu_k $ is a positive measure.  Moreover, $u_k$ and $v_k$ belong to $W^{1,2}_{loc}(\Omega)\bigcap L^\infty(\Omega)$ (see \cite{HKM}, Corollary 7.20). 

Let $h \in C^{\infty}_0 (\Omega)$.  We apply (\ref{by-parts}) with $g=  h/u_k, \rho=\mu_k$, and 
$r=u_k$. Note that $u_k\ge 1$, $g $  is compactly supported since $h$ is, and $ g \in W^{1,2} (\Omega)$ since $u_k \in W^{1,2}_{loc} (\Omega)$.  Then by (\ref{by-parts}),
\begin{equation*}
\begin{aligned}
\int_\Omega \frac {h}{u_k} \, d \mu_k & =  \int_{\Omega} \nabla \left(\frac {h}{u_k}\right) \cdot \nabla u_k \, dx
=  \int_\Omega \frac {\nabla h}{u_k}  \cdot \nabla u_k \, dx - \int_\Omega \frac {|\nabla u_k|^2}{u_k^2} 
h \, dx \\ 
& =  \int_\Omega \nabla h  \cdot \nabla v_k \, dx - \int_\Omega |\nabla v_k|^2 \,  h \, dx.  
\end{aligned}
\end{equation*}

Since $u$ is superharmonic, $u$ is lower semi-continuous, so the set $\{ x \in \Omega: u(x) > e^k \} \equiv \{u>e^k\}$ is open, hence the measure $\mu_k = - \Delta u_k$ is supported on the set $\{u \le e^k\}$ where $u=u_k$.  Hence  $u=u_k$ $d \mu_k$-a.e.,  and 
\begin{equation*}
\begin{aligned}
\left  | \int_\Omega \frac h u \, d \mu - \int_{\Omega} \frac{h}{ u_k} \, d \mu_k  \right | & \le e^{-k} \int_{\{u \geq e^k\}} |h| d \mu +  
e^{-k}  \int_{ \{u =  e^k\}} |h|  \, d \mu_k \\ & 
\le e^{-k} \int_{\Omega} |h| d \mu +  
e^{-k}   \int_{\Omega}| h | \, d \mu_k 
\to 0
\end{aligned}
\end{equation*}
as $k \to \infty$.
Hence
\[
\int_\Omega h q \, dx  = \int_\Omega \frac{h}{u} \,\,  qu \, dx = - \int_\Omega \frac{h}{u}  \,\,  \triangle u \, dx =  \int_\Omega \frac{h}{u}  \, d \mu
= \lim_{k \to \infty} \,  \int_\Omega \frac {h}{u_k} \, d \mu_k.  
\]
Notice that $\nabla v_k = \nabla v$ a.e. on $\{ v<k\}$, and $\nabla v_k = 0$ a.e. on $\{ v\ge k\}$ (see \cite{MZ}, Corollary 1.43). Hence, 
\begin{equation*}
\begin{aligned}
\lim_{k \rightarrow \infty} \int_\Omega \nabla h  \cdot \nabla v_k \, dx  &= \int_\Omega \nabla h  \cdot \nabla v \, dx , \\  
\lim_{k \rightarrow \infty} \int_\Omega |\nabla v_k|^2 \,  h \, dx & = \int_\Omega |\nabla v|^2 \,  h \, dx
\end{aligned}
\end{equation*}
by the dominated convergence theorem. Passing to the limit as $k \to \infty$ in the equation above,  we obtain 
\[
\int_\Omega h q\, dx = \int_\Omega \nabla h \cdot \nabla v \, dx - \int_\Omega |\nabla v|^2 \, h \, dx = -\int_{\Omega} v \triangle h \, dx - \int_\Omega |\nabla v|^2 \, h \, dx , 
\]
which justifies equation (\ref{ric-eq-1}). 

The Riesz decomposition theorem states that a superharmonic function $w$ can be written uniquely as $G(-\triangle w) + g$, where $-\triangle w$, understood in the distributional sense, is called the Riesz measure associated with $w$ and $g$ is the greatest harmonic minorant of $w$ (see 
\cite{AG}, Sec. 4.4).  Hence
\begin{equation}\label{integral-form}
v = G(-\triangle v) + g =  G(|\nabla v|^2 +q) + g,
\end{equation}
where $g$ is the greatest harmonic minorant of $v$.  Since $v \geq 0$, a harmonic minorant of $v$ is $0$, so $g \ge 0$.  It follows from (\ref{integral-form}) and  $u = G(uq) + 1$ that 
$$
g \le v=\log u = \log \left (G (u q) + 1\right)\le G (u q). 
$$
Since $G(uq)$ is a Green potential, the greatest harmonic minorant of $G (uq)$ is $0$,  therefore  $g=0$.  

Hence we have $v =  G(|\nabla v|^2 +q) $, which we have noted (see (\ref{integralformmeasure})) is equivalent to $v$ being a very weak solution of (\ref{nonlineareqn-1}).  

Conversely, suppose $v\in W^{1,2}_{loc}(\Omega)$ is a very weak solution of equation (\ref{nonlineareqn-1}), that is, $v = G (|\nabla v|^2 + q)$.  Then $v \geq 0$.  Let  $v_k = \min\, (v, \, k)$ and $\nu_k = -\Delta v_k$, for $k=1,2, \ldots$ . Then   $v_k\in W^{1,2}_{loc}(\Omega)\bigcap L^\infty(\Omega)$ is superharmonic, and 
\begin{equation}\label{just1}
-\Delta v_k = |\nabla v_k|^2 +q \, \chi_{\{v<k\}} +\tilde \nu_k,
 \end{equation}
where $\tilde \nu_k$ is a nonnegative measure in $\Omega$ supported on $\{v=k\}$. 

Let $u = e^v \geq 1$.  Let $u_k=e^{v_k}$ and $\mu_k=-\Delta u_k$. Since $u_k\in W^{1,2}_{loc}(\Omega)\bigcap L^\infty(\Omega)$, 
it is easy to 
see that 
\begin{equation}\label{just2}
\mu_k=-\Delta u_k = -\Delta v_k \, e^{v_k} -|\nabla v_k|^2  \, e^{v_k} \ge 0.  
\end{equation}
Equation (\ref{just2}) is justified by using integration by parts 
(\ref{by-parts}) with  $g=h e^{v_k}$ where  $h \in C^\infty_0(\Omega)$, and 
 $v_k$ in place of $r$: 
\begin{equation*}
\begin{aligned}
\int_\Omega h \, e^{v_k} \, d \nu_k & = 
\int_\Omega \nabla (h \, e^{v_k}) \cdot \nabla v_k \, dx 
\\ & = \int_\Omega  e^{v_k} \, \nabla h \cdot  \nabla v_k  \, dx  + \int_\Omega  h \, |\nabla v_k|^2 \, e^{v_k} \,dx \\ 
& = \int_\Omega \, \nabla h \cdot  \nabla u_k  \, dx  + \int_\Omega  h \, |\nabla v_k|^2 \, e^{v_k} \,dx \\
& 
 =  \int_\Omega h  \, d \mu_k  + \int_\Omega  h \, |\nabla v_k|^2 \, e^{v_k} \,dx. 
\end{aligned}
\end{equation*}
Hence, 
\begin{equation*}
\begin{aligned}
\langle h, \mu_k\rangle  = \int_\Omega \nabla h \cdot  \nabla (e^{v_k}) \, 
 dx 
& = \int_\Omega e^{v_k} \nabla h  \cdot  \nabla v_k \, dx \\ 
 & = \int_\Omega h \, e^{v_k} \, d \nu_k 
- \int_\Omega    h  \, | \nabla v_k |^2 \, e^{v_k}  \,  dx   \\ 
& =\int_\Omega h \,  e^{v_k} \  \chi_{\{v<k\}}  q\, d x + \int_\Omega h \,  e^{v_k} \   d \tilde \nu_k ,
\end{aligned}
\end{equation*}
where in the last expression we used (\ref{just1}). From the preceding estimates it follows that 
$
\langle h, \mu_k\rangle \ge 0$ if $h \ge 0$, and consequently $u_k$ is superharmonic, and 
\begin{equation}\label{just3}
-\Delta u_k \ge q  u_k \, \chi_{\{u_k<e^k\}}. 
\end{equation}

Clearly,  $u=e^v<+\infty$-a.e., and  $u= \lim 
_{k \to +\infty} u_k$ is superharmonic in $\Omega$ as the limit of the 
increasing sequence of superharmonic functions $u_k$. Since $\mu_k \to \mu$ in the sense of measures, where $\mu = - \triangle u$, (\ref{just3}) yields 
\begin{equation}\label{just4}
-\Delta u \ge  qu \quad \text{in} \, \, \Omega 
\end{equation}
 in the sense of measures, where $q u \in L^1_{loc}(\Omega)$. 

It follows from (\ref{just4}) that 
 $\omega = - \Delta u-qu$ is a non-negative measure in $\Omega$,  so by the Riesz decomposition theorem
$$
u = G(-\Delta u) + g =  G(qu) + G \omega+ g \geq G(qu) + g,
$$
where $g$ is the greatest harmonic minorant of $u$. Since $u \ge 1$, i.e., 
$1$ is a harmonic minorant of $u$, it follows that $g \ge 1$, and consequently, 
\begin{equation}\label{iter}
u \ge G(qu ) + 1 = Tu + 1,
\end{equation}
for $T$ defined by (\ref{operator-tf}).  Since $u \ge Tu$, it follows by Schur's test that 
$||T||_{L^2(\Omega, q dx) \to L^2(\Omega, q dx)} \le 1$, and hence 
 (\ref{impliesnormcond2}) 
holds with $\beta =1$ by Lemma~\ref{leastconst}.

Iterating (\ref{iter}) and taking the limit, we see that 
$$
\phi \equiv 1 + \mathcal{G} q = 1 + \sum_{j=1}^{\infty} G_j q = 1 + \sum_{j=1}^{\infty} T^j 1  \le u < +\infty \, \, \text{a.e.},
$$
and 
$$
\phi = G(q \phi ) + 1.
$$
Hence $\phi$ is a positive very weak solution of (\ref{dirichlet}).  Thus (\ref{balayagecond}) holds, by Theorem \ref{nonlinthm} (ii).
 \end{proofof}

\begin{Rem} 1. We remark in conclusion that the main results of this paper 
remain valid for any elliptic operator $\mathcal{L}$ whose Green's function $G^{\mathcal{L}}$ is equivalent to the Green's function $G$ of the 
Laplacian (see \cite{An}). 

2. Our main results also hold for general locally finite Borel measures $\omega$ in $\Omega$ in place of $q \in L^1_{loc}(\Omega)$, with minor adjustments in the proofs. Notice that condition (\ref{impliesnormcond2}) for  $\omega$ in place of $q dx$ implies 
that $\omega$ is absolutely continuous with respect to capacity 
(see \cite{M}), and all solutions considered in this paper are superharmonic, i.e., finite quasi-everywhere in $\Omega$; moreover, they actually lie in $W^{1, 2}_{loc} (\Omega)$ (see \cite{JMV}, Theorem 6.2).

3. Concerning Theorem  \ref{nonlinthm} and Theorem \ref{riccatithm}, 
suppose \eqref{impliesnormcond2} holds with $\beta <1$. Then a  necessary and sufficient condition in order that  $w=u_{1} - 1  \in L^{1, 2}_0(\Omega)$ 
is  $\int_\Omega G q \, q \, dx < \infty$, i.e., $q \in L^{-1,2} (\Omega)$. 
\end{Rem} 

The sufficiency part  of the last statement follows from  the Lax-Milgram Lemma since $-\triangle w = q w + q$
where $q \in L^{-1,2} (\Omega)$; necessity is a consequence of the fact that $w=G(w q+q)$, so that 
\begin{equation*}
\begin{aligned}
\int_\Omega |\nabla w|^2 dx & = \int_\Omega |\nabla G(w q+q)|^2 dx \\
& = \int_\Omega G(w q+q) \, (wq+q) dx \ge \int_\Omega G q \, q dx . 
 \end{aligned}
 \end{equation*}
In particular, if 
\eqref{impliesnormcond2} holds,  and 
$q \in L^1(\Omega)$, then for all $h \in C^\infty_0(\Omega)$,
\begin{equation*}
\begin{aligned}
\left |\int_\Omega h \, q \, dx \right | & \le \Big(\int_\Omega h^2 \, q \, dx\Big)^{1/2} ||q||_{L^1(\Omega)}^{1/2} \\ 
 & \le \beta^{1/2} ||\nabla h||_{L^2 (\Omega)} 
 ||q||_{L^1(\Omega)}^{1/2}< \infty . 
 \end{aligned}
 \end{equation*}
Hence, by duality $q \in L^{-1,2} (\Omega)$, and 
consequently $w=u_{1} - 1  \in L^{1, 2}_0(\Omega)$,  
for all $n \ge 2$ (see also \cite{AB}, \cite{ADP}). 

This also gives a weak solution $v \in L^{1, 2}_0(\Omega)$ to 
\eqref{nonlineareqn-1} such that $e^v-1 \in L^{1, 2}_0(\Omega)$, as in 
 \cite{FM1},  \cite{FM3}, if  
$q \ge 0$. For  arbitrary distributions $q \in L^{-1,2} (\Omega)$,  
$w=u_{1} - 1  \in L^{1, 2}_0(\Omega)$ and 
$v =\log u_{1} \in L^{1, 2}_0(\Omega)$ is a weak solution to 
\eqref{nonlineareqn-1}, provided $q$ is form bounded with the upper form  bound strictly less than $1$ (see \cite{JMV}).

\end{document}